\documentclass[11pt,reqno]{amsart}
\pdfoutput=1 % needed for arXiv

\usepackage[english]{babel}
\usepackage[utf8]{inputenc}
\usepackage[T1]{fontenc}
\usepackage{lmodern} \normalfont %to load T1lmr.fd 
\DeclareFontShape{T1}{lmr}{bx}{sc} { <-> ssub * cmr/bx/sc }{}
\usepackage{microtype}	% for improved spacing between words and letters

\usepackage{amssymb}
\usepackage{amsmath}
\usepackage{comment}
\usepackage{amsthm}
\usepackage{algorithm}
\usepackage{algpseudocode}
\usepackage{mathtools}
\usepackage{mathrsfs}
\usepackage{accents} % correct spacing for accents in sub- and superscript

\usepackage{subdepth} % correct positions for sub- and superscripts
\usepackage{etoolbox}
\usepackage{siunitx}
\usepackage{dsfont} % for special 1, \mathds{1}
\usepackage{bm}

\usepackage{tabularx}

\usepackage{graphicx}
\usepackage{color}
\usepackage[dvipsnames]{xcolor}
\usepackage{circuitikz}
\usepackage{tikz}
\usetikzlibrary{external}
\usetikzlibrary{spy}
\usetikzlibrary{calc,positioning,shapes}
\usetikzlibrary{patterns,decorations.pathmorphing,decorations.markings}
\usetikzlibrary{external}
\usepackage{pgfplots}
\pgfplotsset{compat=newest}
\usepackage[margin=10pt,font=small,labelfont=bf,labelsep=endash]{caption}
\usepackage{subcaption}

\usepackage{booktabs}
\usepackage{paralist}
\usepackage{soul}
\usepackage{listings}

\numberwithin{equation}{section}

\usepackage{algorithm}
\usepackage{algorithmicx}
\usepackage{algpseudocode}

\usepackage{cite}
\usepackage{multirow} % for multirows in tabular environment
\usepackage{enumitem} % for labelling items in enumerate
\setlist[enumerate]{label=(\roman*)}
\usepackage{xspace}

\usepackage[colorlinks,linkcolor=teal,citecolor=teal,urlcolor=teal]{hyperref}
\usepackage[nameinlink,noabbrev]{cleveref}

\theoremstyle{plain}
\newtheorem{theorem}{Theorem}[section]
\newtheorem{proposition}[theorem]{Proposition}
\AddToHook{env/proposition/begin}{\crefalias{theorem}{proposition}}
\Crefname{proposition}{Proposition}{Propositions}

\newtheorem{lemma}[theorem]{Lemma}
\AddToHook{env/lemma/begin}{\crefalias{theorem}{lemma}}
\Crefname{lemma}{Lemma}{Lemmata}

\AddToHook{env/corollary/begin}{\crefalias{theorem}{corollary}}
\Crefname{corollary}{Corollary}{Corollaries}

\newtheorem{remark}[theorem]{Remark}
\AddToHook{env/remark/begin}{\crefalias{theorem}{remark}}
\Crefname{remark}{Remark}{Remarks}

\newtheorem{definition}[theorem]{Definition}
\AddToHook{env/definition/begin}{\crefalias{theorem}{definition}}
\Crefname{definition}{Definition}{Definitions}

\newtheorem{assumption}[theorem]{Assumption}
\AddToHook{env/assumption/begin}{\crefalias{theorem}{assumption}}
\Crefname{assumption}{Assumption}{Assumptions}

\AddToHook{env/example/begin}{\crefalias{theorem}{example}}
\Crefname{example}{Example}{Examples}

\AddToHook{env/problem/begin}{\crefalias{theorem}{problem}}
\Crefname{problem}{Problem}{Problems}

\usepackage{ifthen}
\newboolean{preprint}
\setboolean{preprint}{false} 	% true for preprint, false for submitted version

\newcommand{\N}{\ensuremath\mathbb{N}}

\newcommand{\R}{\ensuremath\mathbb{R}}
\newcommand{\C}{\ensuremath\mathbb{C}}

\newcommand{\smoothFunctions}[3][]{\ifthenelse{\equal{#1}{}}{\mathcal{C}^{#2}}{\mathcal{C}_{#1}^{#2}}(#3)}

\newcommand{\dist}[1][]{\ifthenelse{\equal{#1}{}}{\mathbb{D}}{#1_{\mathbb{D}}}}

\newcommand{\GL}[1]{\mathrm{GL}_{#1}}

\newcommand{\Spsd}[1]{\mathcal{S}^{#1}_{\succeq}}

\newcommand{\Frob}{\mathrm{F}}

\newcommand{\T}{\ensuremath\mathsf{T}}

\newcommand{\integrate}[1]{\mathrm{d}#1}

\newcommand{\ds}{\integrate{s}}
\newcommand{\dt}{\integrate{t}}

\DeclareMathOperator{\ee}{e}

\DeclareMathOperator{\Kern}{ker}

\DeclareMathOperator{\image}{img}
\newcommand{\img}{\image}

\DeclareMathOperator{\spann}{span}

\DeclareMathOperator{\vec2}{vec}

\newcommand{\calB}{\mathcal{B}}
\newcommand{\calC}{\mathcal{C}}

\newcommand{\calG}{\mathcal{G}}

\newcommand{\calJ}{\mathcal{J}}

\newcommand{\calL}{\mathcal{L}}
\newcommand{\calM}{\mathcal{M}}
\newcommand{\calN}{\mathcal{N}}
\newcommand{\calO}{\mathcal{O}}
\newcommand{\calP}{\mathcal{P}}
\newcommand{\calQ}{\mathcal{Q}}
\newcommand{\calR}{\mathcal{R}}
\newcommand{\calS}{\mathcal{S}}
\newcommand{\calT}{\mathcal{T}}
\newcommand{\calU}{\mathcal{U}}
\newcommand{\calV}{\mathcal{V}}
\newcommand{\calW}{\mathcal{W}}

\newcommand{\stx}{\ensuremath{\bm{x}}}
\newcommand{\state}{\stx}
\newcommand{\stateDim}{n}
\newcommand{\stateRed}{\reduce{\state}}
\newcommand{\stateDimRed}{r}
\newcommand{\inp}{\ensuremath{\bm{u}}}
\newcommand{\inpDim}{m}
\newcommand{\out}{\ensuremath{\bm{y}}}
\newcommand{\outDim}{p}

\newcommand{\fa}{\ensuremath{\bm{a}}}
\newcommand{\fb}{\ensuremath{\bm{b}}}

\newcommand{\fe}{\ensuremath{\bm{e}}}

\newcommand{\fp}{\ensuremath{\bm{p}}}

\newcommand{\fu}{\ensuremath{\bm{u}}}
\newcommand{\fv}{\ensuremath{\bm{v}}}
\newcommand{\fw}{\ensuremath{\bm{w}}}
\newcommand{\fx}{\ensuremath{\bm{x}}}

\newcommand{\fz}{\ensuremath{\bm{z}}}

\newcommand{\fA}{\ensuremath{\bm{A}}}
\newcommand{\fB}{\ensuremath{\bm{B}}}
\newcommand{\fC}{\ensuremath{\bm{C}}}
\newcommand{\fD}{\ensuremath{\bm{D}}}
\newcommand{\fE}{\ensuremath{\bm{E}}}

\newcommand{\fG}{\ensuremath{\bm{G}}}

\newcommand{\fI}{\ensuremath{\bm{I}}}
\newcommand{\fJ}{\ensuremath{\bm{J}}}
\newcommand{\fK}{\ensuremath{\bm{K}}}

\newcommand{\fM}{\ensuremath{\bm{M}}}
\newcommand{\fN}{\ensuremath{\bm{N}}}

\newcommand{\fP}{\ensuremath{\bm{P}}}

\newcommand{\fR}{\ensuremath{\bm{R}}}
\newcommand{\fS}{\ensuremath{\bm{S}}}
\newcommand{\fT}{\ensuremath{\bm{T}}}
\newcommand{\fU}{\ensuremath{\bm{U}}}
\newcommand{\fV}{\ensuremath{\bm{V}}}
\newcommand{\fW}{\ensuremath{\bm{W}}}
\newcommand{\fX}{\ensuremath{\bm{X}}}

\newcommand{\fZ}{\ensuremath{\bm{Z}}}

\newcommand{\flambda}{\ensuremath{\bm{\lambda}}}

\newcommand{\frho}{\ensuremath{\bm{\rho}}}

\newcommand{\fPi}{\ensuremath{\bm{\Pi}}}
\newcommand{\fSigma}{\ensuremath{\bm{\Sigma}}}

\newcommand{\fcP}{\mathbf{\calP}}
\newcommand{\fcQ}{\mathbf{\calQ}}
\newcommand{\zeroVec}{\mathbf{0}}
\newcommand{\zeroMat}{\mathbf{0}}

\newcommand{\difW}{\hat{\mathbf{V}}}
\newcommand{\impW}{\hat{\mathbf{W}}}

\newcommand{\fcB}{\mathbf{\calB}}
\newcommand{\fcC}{\mathbf{\calC}}

\newcommand{\switch}{q}
\newcommand{\switchingSet}{\calJ}
\newcommand{\diff}{\mathrm{diff}}
\newcommand{\imp}{\mathrm{imp}}
\newcommand{\indDAE}{\nu}
\newcommand{\tol}{\texttt{tol}}

\newcommand{\system}{\Sigma}
\newcommand{\reduce}[1]{\tilde{#1}}
\newcommand{\systemRed}{\reduce{\system}}
\newcommand{\switchedSys}{\system_{\switch}}
\newcommand{\switchedSysRed}{\systemRed_{\switch}}
\newcommand{\reduceBis}[1]{\hat{#1}}
\newcommand{\switchedSysRedJumps}{\hat \system_{\switch}}

\newcommand{\gleMat}[1]{\mathbf{\mathscr{#1}}}

\newcommand{\lyapOper}{\calL}
\newcommand{\gleOper}{\Pi}
\newcommand{\tfinal}{t_{\mathrm{f}}}

\newcommand{\indKS}{\ell}

\newcommand{\SwDifState}{\hat{\fz}}
\newcommand{\SwDifStateRed}{\tilde{\fz}}
\newcommand{\SwDifStateRedBis}{\bar{\fz}}
\newcommand{\StateTran}{\fP}

\newcommand{\stateBal}{\bar{\state}}
\newcommand{\bfX}{\mathbf{x}}
\newcommand{\gleG}{\gleMat{G}}
\newcommand{\gleH}{\gleMat{H}}
\newcommand{\matlab}{{\sc Matlab}}

\newcommand{\lineWidth}{1.2pt}
\newcommand{\imageWidth}{2.0in}
\newcommand{\imageHeight}{1.8in}
\definecolor{mycolor1}{rgb}{0.00000,0.44700,0.74100}% blue
\definecolor{mycolor2}{rgb}{0.85000,0.32500,0.09800}% red
\definecolor{mycolor3}{rgb}{0.92900,0.69400,0.12500}% orange/yellow
\definecolor{mycolor4}{rgb}{0.46600,0.67400,0.18800}% green
\definecolor{mycolor5}{rgb}{0.49400,0.18400,0.55600}% purple

\newcommand{\abbr}[1]{\textsf{#1}\xspace}
\newcommand{\FOM}{\abbr{FOM}}
\newcommand{\GLE}{\abbr{GLE}}
\newcommand{\GLEs}{\abbr{GLEs}}

\newcommand{\LMIs}{\abbr{LMIs}}
\newcommand{\ROM}{\abbr{ROM}}
\newcommand{\ROMs}{\abbr{ROMs}}
\newcommand{\MOR}{\abbr{MOR}}
\newcommand{\SDAE}{switched~\abbr{DAE}}
\newcommand{\SDAEs}{switched~\abbr{DAEs}}
\newcommand{\DAE}{\abbr{DAE}}
\newcommand{\DAEs}{\abbr{DAEs}}
\newcommand{\QWF}{\abbr{QWF}}
\newcommand{\ODE}{\abbr{ODE}}
\newcommand{\ODEs}{\abbr{ODEs}}

\newcommand{\SVD}{\abbr{SVD}}
\newcommand{\SLS}{\abbr{SLS}}
\newcommand{\SLSs}{\abbr{SLSs}}
\newcommand{\PBR}{\abbr{PBR}}

\title{Balancing-based model reduction for switched descriptor systems}
\author[M.~Manucci \and B.~Unger]{Mattia Manucci${}^\star$ \and Benjamin Unger${}^\star$}

\address{${}^{\star}$ Institute for Applied and Numerical Mathematics, Karlsruhe Institute of Technology, 76131 Karlsruhe, Germany}
\email{\{mattia.manucci,benjamin.unger\}@kit.edu}

\newcommand{\ourKeywords}{model order reduction, differential-algebraic equations, switched systems, balanced truncation, error bound, generalized Lyapunov equation}
\date{\today}
\begin{document}
	
\begin{abstract}
	We propose a novel projection-based model order reduction (MOR) algorithm for a broad class of switched linear descriptor systems. Our approach integrates the reformulation strategy of [Hossain \& Trenn, DAE Panel, 2024], which converts the switched descriptor system into a switched ordinary differential equation featuring state- and input-dependent jumps as well as impulsive components at the switching instants in the output, with the piecewise balanced truncation framework for switched linear systems introduced in [Manucci \& Unger, arXiv:2601.13039, 2026]. 

The central idea is to additionally reformulate the switched system with input jumps and impulsive output as a standard switching system that exhibits only state-dependent jumps, while appropriately augmenting the input and output dimensions to preserve the original input–output behavior. Based on this reformulation, we develop new stopping criteria for the stationary iterative scheme employed to approximate the solution of a generalized Lyapunov equation, a key ingredient for the computation of the projection matrices in our MOR procedure. 

Finally, numerical experiments are conducted to demonstrate and assess the effectiveness of the proposed reduction methodology.
\end{abstract}
\maketitle
	
%
%30E05 = Moment problems and interpolation problems in the complex plane
%37M99 = Dynamical systems and ergodic theory
%65F45 = Numerical methods for matrix equations
%65F55 = Numerical methods for low-rank matrix approximation; matrix compression
%65P99 = Numerical problems in dynamical systems
%65L80 = Numerical methods for differential-algebraic equations 
%93A30 = Mathematical modeling (models of systems, model-matching, etc.)
%93A15 = Large-scale systems
%93B99 = Systems theory
%
	
%-----------------------------------------------------------------------------%
\section{Introduction}
We consider systems of switched \emph{differential-algebraic equations} (\DAE) of the form
\begin{equation}
	\label{eqn:sDAE}
	\switchedSys \quad \left\{\quad \begin{aligned}
		\fE_{\switch(t)}  \dot{\stx}(t)  &= \fA_{\switch(t)} \stx(t) +\fB_{\switch(t)}\inp(t), & \stx(t_0) &= \zeroVec, \\
		\out(t) &= \fC_{\switch(t)}\stx(t),\\
	\end{aligned}\right.
\end{equation}
where $\switch\colon \R\to \switchingSet\vcentcolon=\{1,\ldots,M\}$ is the external switching signal, which we assume to be an element of the set of allowed switching signals
\begin{equation}
	\label{eqn:suitableSwitchingSignals}
	\calS \vcentcolon= \{\switch\colon \R\to \switchingSet \mid \switch \text{ is right continuous with locally finite number of jumps}\}.
\end{equation}
The symbols $\stx(t)\in\R^{\stateDim}$, $\inp(t)\in\R^{\inpDim}$, and $\out(t)\in\R^{\outDim}$ denote the \emph{state}, the controlled \emph{input}, and the measured \emph{output} at time $t$,  respectively. The system matrices $\fE_{j}\in\R^{\stateDim\times \stateDim}$, $\fA_{j}\in\R^{\stateDim\times \stateDim}$, $\fB_j\in \R^{\stateDim\times \inpDim}$, and $\fC_j\in\R^{\outDim\times \stateDim}$ correspond to the \DAE active in mode $j\in\switchingSet$. We emphasize that $\fE_{j}$ might be singular, and we assume that the finite eigenvalues of the matrix pair $(\fE_j,\fA_j)$ have negative real part for all $j\in\calJ$. We refer to~\eqref{eqn:sDAE} as the \emph{full-order model} (\FOM). Sample applications include robot manipulators, traffic management, automatic gear shifting, and power systems; see, for instance, \cite{Che04} and the references therein. 

If \eqref{eqn:sDAE} must be evaluated repeatedly, as, e.g., in simulations arising from an optimal control problem \cite{KarMUV24}, for varying inputs or switching signals, or when matrix equations or inequalities in a synthesis setting need to be solved, then a large state dimension $\stateDim$ makes this a computationally demanding or even infeasible task. In such scenarios, one can rely on \emph{model order reduction} (\MOR) and replace \eqref{eqn:sDAE} with the \emph{reduced-order model} (\ROM)
\begin{equation}
	\label{eqn:sDAE:ROM}
	\switchedSysRed \quad \left\{\quad \begin{aligned}
		\reduce{\fE}_{\switch(t)}  \dot{\stateRed}(t)  &= \reduce{\fA}_{\switch(t)} \stateRed(t) + \reduce{\fB}_{\switch(t)}\inp(t), &	\stateRed(t_0) &= \zeroVec, \\
		\reduce{\out}(t) &= \reduce{\fC}_{\switch(t)}\stateRed(t),\\
	\end{aligned}\right.
\end{equation}
with $\reduce{\fE}_{j}, \reduce{\fA}_{j}\in\R^{\stateDimRed\times \stateDimRed}$, $\reduce{\fB}_j\in \R^{\stateDimRed\times \inpDim}$, and $\reduce{\fC}_j\in\R^{\outDim\times \stateDimRed}$, and $\stateDimRed\ll \stateDim$. In many cases, see, for instance, \cite{Ant05}, the reduced system matrices are obtained via Petrov-Galerkin projection; i.e., one constructs matrices $\fV,\fW\in\R^{\stateDim\times\stateDimRed}$ and then defines
\begin{align}
	\label{eqn:sDAE:ROM:matrices}
	\reduce{\fE}_j &\vcentcolon= \fW^\T \fE_j \fV, &
	\reduce{\fA}_j &\vcentcolon= \fW^\T \fA_j \fV, &
	\reduce{\fB}_j &\vcentcolon= \fW^\T \fB_j, &
	\reduce{\fC}_j &\vcentcolon= \fC_j \fV.
\end{align}
The goal of \MOR is thus to derive the matrices $\fW,\fV$ in a computationally efficient and robust way, such that the error $\out-\reduce{\out}$ is small in some given norm.

\begin{remark}
In this work, we focus on system modes that have the same dimension $\stateDim$. Modes with differing dimensions can also be handled, as long as appropriate transition matrices from mode $i$ to mode $j$ are specified for every $i,j\in\switchingSet$.
\end{remark}

%-----------------------------------------------------------------------------%
\subsection{Main contributions}
This work extends balancing-based model order reduction to switched descriptor systems with arbitrary externally prescribed switching signals of the form~\eqref{eqn:suitableSwitchingSignals}. The proposed methodology builds upon the recently introduced \emph{piecewise balancing reduction} (\PBR) framework for switched linear systems~\cite[Sec.~4]{ManU26}. Its key ingredient is a reformulation of the switched descriptor system that is amenable to the \PBR framework.
The main contributions of this paper are twofold.
\begin{enumerate}
	\item Motivated by \cite{Hos22,HosT23}, we derive in \Cref{Sec:NEW-FORMULATION} an equivalent reformulation of the \SDAE~\eqref{eqn:sDAE}, which we subsequently modify in \Cref{eqn:ext:reach,eqn:ext:obs} with augmented input and output matrices. Based on these modified systems, we characterize reachability and observability in \Cref{teo:reachability,teo2bis}.
	\item Building on these characterizations, we develop in \Cref{sec:PBR-sDAE} the first projection-based balancing method for switched descriptor systems with arbitrary switching signals of the form~\eqref{eqn:suitableSwitchingSignals}. The proposed method extends \PBR, including its a posteriori error certification, to this broader class of systems.
\end{enumerate}

The construction of the projection spaces requires the solution of \emph{generalized Lyapunov equations} (\GLEs). To this end, we build on the methodology in \cite[Sec.~3]{ManU26} and derive a relative stopping criterion (see \Cref{teo1}) that enables control of the relative error in the approximation of the \GLE solution.
The effectiveness of the proposed methodology is substantiated through numerical experiments conducted on a constrained mass--spring--damper system \Cref{subsec:MSD}, as well as on an artificial \SDAE constructed by alternately switching between the constrained mass–spring–damper system and a semidiscretized Stokes system \Cref{subsec:Stokes}.

%-----------------------------------------------------------------------------%
\subsection{Literature review and state-of-the-art}
Our method is mainly inspired by \cite{ManU26} and \cite{HosT23}, where \cite{HosT23} is based on the thesis \cite{Hos22}. The latter two publications appear to be the only references for \MOR for switched descriptor systems. Nevertheless, \cite{HosT23} is restricted to \SDAE where the switching sequence is known a priori and no error certification is provided. Let us also mention \cite{SajCZS13}, where a dimensionality reduction of a switched descriptor system is realized via one of the Wong sequences defined in \eqref{eqn:WongSequences}, which is subsequently used to obtain stability conditions \cite{SajCZS19}. 

The core idea of \cite{Hos22} is the reformulation of the \SDAE~\eqref{eqn:sDAE} as a switched linear system of \emph{ordinary differential equations} (\ODEs) with jumps and impulses, and subsequently application of the midpoint-based balanced truncation introduced in \cite{HosT24} on this system. The literature of \MOR for switched \ODEs involves several different approaches. In particular, the authors of \cite{PapP14,PapP16} propose constructing the \ROM for each mode independently of the other modes, completely ignoring the transition from one mode to another. Moreover, if a state-dependent switching signal is allowed, any approximation of the \SDAE may be arbitrarily poor. Both phenomena are detailed with examples in \cite{SchU18}. In \cite{WuZ09}, an approach based on a set of coupled linear matrix inequalities is proposed, which becomes infeasible in a large-scale context. Nevertheless, the matrix inequalities can be used to guarantee quadratic stability of the reduced system and to derive an error bound \cite{PetWL13}. Instead of matrix inequalities, \cite{ShaW12,GosPAF18} propose solving a set of coupled Lyapunov equations to compute Gramians, which can then be used for a balancing-based model reduction. There is no guarantee that this approach yields a solution \cite{Lib03}, and in large-scale settings, the computational complexity may be prohibitive. If the Gramians for each mode can be simultaneously diagonalized, then classical balanced truncation methods can be adapted, as discussed in \cite{MonTC12}.
In contrast, the methods reported in \cite{SchU18,PonGB20} rely on reformulating the switched system as a non-switched system by suitably interpreting the switching signal as a control input. In \cite{SchU18}, the switched system is recast as a linear system so that standard methods can be applied, whereas in \cite{PonGB20}, it is recast as a bilinear system and a balanced truncation approach is employed. In this paper, we closely follow the second strategy. We emphasize that both methods provide a-priori error certification if certain non-generic conditions hold. If prior information about the switching sequence is available, then the method in \cite{PonGB20} can be further specialized, as reported in \cite{GosPBA20}. The \MOR method introduced in \cite{ManU26}, which relies on piecewise constant-in-time projection matrices obtained through a balancing procedure, circumvents the non-generic conditions by establishing an a posteriori error bound that often can be evaluated efficiently. Interpolation-based techniques are discussed in \cite{ScaA16} for hybrid and in \cite{BasPWL16} for switched systems. For a data-driven approach to obtaining reduced models, we refer the reader to \cite{GosPA18}. Let us emphasize that model reduction is closely related to realization theory. For switched systems, the associated connections are illustrated in \cite{PetG22} and the references therein. We conclude by mentioning \cite{PeiK19}, where the authors discretize the control variable to obtain a switched system of autonomous partial differential equations, which they approximate with \ROMs in a model predictive control framework.

%-----------------------------------------------------------------------------%
\subsection{Organization of the manuscript}
After this introduction, \Cref{sec:preliminaries} is dedicated to presenting the main mathematical tools required for the remainder of the work. \Cref{sec:MORforDAE} forms the central part of the paper, where we demonstrate how a \ROM for a \SDAE can be obtained by first deriving a \ROM for a suitably defined switched \ODE with state jumps. In \Cref{sec:numericalDetails}, we examine several numerical aspects that are essential for our \MOR procedure. In addition, we derive a novel stopping criterion for a rigorous bound on the relative error in approximating the solution of certain matrix equations, which are required for computing the projection basis used in \MOR. Finally, in \Cref{sec:examples}, we demonstrate the effectiveness of the proposed reduction procedure by means of numerical experiments.
%-----------------------------------------------------------------------------%
\subsection{Notation}
The symbols $\fI$ and $\GL{\stateDim}$ denote the identity matrix of appropriate size and the set of $\stateDim\times \stateDim$ real nonsingular matrices, respectively. Let $\fA\in\R^{\stateDim\times\stateDim}$. Then we call~$\fA$ asymptotically stable if all eigenvalues of $\fA$ are contained in the open left-half complex plane. Moreover, we write $\fA\succeq \zeroMat$ or $\fA \preceq \zeroMat$ if $\fA$ is positive or negative semidefinite, respectively. The preimage of $\fM\in\R^{\stateDim_1\times\stateDim_2}$ with respect to a linear subspace $\calN \subseteq \R^{\stateDim_1}$ is denoted with
\begin{equation}\label{eqn:pre:img}
	\fM^{-1}(\calN)\;\vcentcolon=\;\{\fx\in\R^{\stateDim_2} \mid \fM\fx\in\calN\}.
\end{equation}
The smallest $\fA$-invariant subspace containing $\calN$ and the largest $\fA$-invariant subspace contained in $\calN$ are
\begin{align*}
	\langle \fA \mid \calN  \rangle &\vcentcolon= \calN + \fA \calN +\ldots+\fA^{n-1} \calN, &
	\langle \calN \mid \fA  \rangle &\vcentcolon= \calN \cap \fA^{-1}\calN\cap\ldots\cap\fA^{-(n-1)}\calN,
\end{align*} 
where the addition symbol between two sets has to be interpreted as the union of those sets.

%-----------------------------------------------------------------------------%
\section{Preliminaries} \label{sec:preliminaries}
	We begin this section by reviewing key concepts related to \DAEs and \SDAEs in \Cref{subsec:DAE} and \Cref{subsec:sDAE}, respectively. Then, in \Cref{subsec:reachObserv}, we present the notions of reachable and observable sets for general switched systems, thereby encompassing both standard \SLSs and \SDAEs. By standard \SLS, we mean systems of the form \eqref{eqn:sDAE}–\eqref{eqn:suitableSwitchingSignals} for which the matrix $\fE_{\switch(t)}$ is invertible and $\fE_{\switch(t)}^{-1}\fA_{\switch(t)}$ is asymptotically stable for all $t$. In the recent study \cite{ManU26}, a novel projection-based \MOR methodology is proposed for this class of systems. The method utilizes piecewise-defined projection matrices that are constructed from approximations of the solutions to suitably formulated \GLEs. A detailed exposition of this approach is provided in \Cref{subsec:PBR-SLS}.
%-----------------------------------------------------------------------------%
\subsection{Differential-algebraic equations}
\label{subsec:DAE}
If the switching signal $q\in\calS$ is constant, then the \SDAE~\eqref{eqn:sDAE} reduces to a \DAE in control form, which motivates the definition of the non-switched descriptor system
\begin{equation}
	\label{eqn:DAE}
	\Sigma_{\DAE}\quad\left\{\quad \begin{aligned}
		\fE \dot{\stx}(t)  &= \fA \stx(t)+\fB \inp(t),\quad\stx(t_0)=\stx_0, \\
		\out(t) &= \fC \stx(t),\\
	\end{aligned}\right.
\end{equation}
with $\fE$, $\fA\in \R^{\stateDim\times \stateDim}$, $\fB\in\R^{\stateDim\times \inpDim}$, $\fC\in \R^{\outDim\times \stateDim}$, and initial value $\stx_0\in\R^{\stateDim}$. To ensure the existence and uniqueness of solutions of the \DAE~\eqref{eqn:DAE}, the matrix pair $(\fE,\fA)$ must satisfy certain properties; see, for instance, \cite[Cha.~2]{KunM06}. In more detail, we assume that the matrix pair $(\fE, \fA)$ is \emph{regular}, i.e., $\det(s\fE-\fA)\in\C[s]\setminus\{0\}$. In this case, one can show that in the space of piecewise-smooth distributions \cite{Tre09}, the initial trajectory problem associated with the \DAE~\eqref{eqn:DAE} has a unique solution for any initial value and any control input. If a classical solution is required, then, in addition to a regular matrix pair, the initial value and the control input have to satisfy a consistency condition and the control input needs to be sufficiently smooth; cf.~\cite{KunM06} and the forthcoming \Cref{rem:classicalSolution}. Regularity of a matrix pair can be characterized by the Weierstra\ss{} form \cite{Gan59} or the slightly simplified \emph{quasi-Weierstra\ss{} form} (\QWF) \cite{BerIT12}.
	
\begin{theorem}[Quasi-Weierstrass Form, \cite{BerIT12}]\label{teo:QWF}
	A matrix pair $(\fE,\fA)\in\R^{\stateDim\times\stateDim}\times\R^{\stateDim\times\stateDim}$ is regular if and only if there exist matrices $\fS,\fT\in\GL{\stateDim}$ such that
	\begin{equation}
		\label{eqn:QWF}
		\left(\fS\fE\fT, \fS\fA\fT\right) = \Bigg(   \begin{bmatrix}
			\fI &\zeroMat\\
			\zeroMat &\fN
		\end{bmatrix},\begin{bmatrix}
			\fJ&\zeroMat\\
			\zeroMat & \fI
		\end{bmatrix}  \Bigg),
	\end{equation}
	where $\fN\in\R^{\stateDim_{\fN} \times \stateDim_{\fN}}$ is nilpotent with nilpotency index $\indDAE$ and $\fJ\in\R^{\stateDim_{\fJ} \times \stateDim_{\fJ} }$, with $\stateDim_{\fJ}=\stateDim-\stateDim_{\fN}$. %The decoupling \eqref{eqn:QWF} is called quasi-Weierstrass form.
\end{theorem}

\begin{remark}\label{rmk:AS}
	The assumption that the finite eigenvalues of the matrix pair $(\fE,\fA)$ have a negative real part implies that the matrix $\fJ$ appearing in \eqref{eqn:QWF} is asymptotically stable. 
\end{remark}

\ifthenelse{\boolean{preprint}}{
The \QWF allows a decoupling of the \DAE~\eqref{eqn:DAE} into an \emph{ordinary differential equation} (\ODE) and a nilpotent \DAE
\begin{subequations}
	\label{eqn:DAE:decoupled}
	\begin{align}
		\label{eqn:DAE:slow}
		\dot{\fv} &= \fJ\fv+\fB_{\fv}\inp,\\
		\label{eqn:DAE:fast}
		\fN\dot{\fw} &= \fw+\fB_{\fw}\inp,
	\end{align}
\end{subequations}
which can be used to derive an explicit solution formula; see \cite[Cha.~2]{KunM06}. In particular, the fast subsystem~\eqref{eqn:DAE:fast} imposes a so-called \emph{consistency condition} on the initial value $\stx_0$ for a classical solution to exist \cite[Cha.~2]{KunM06}.}{} One way to construct the matrices~$\fS,\fT\in\GL{\stateDim}$ is given via the \emph{Wong sequences} \cite{Won74}, which are defined as 
\begin{subequations}
	\label{eqn:WongSequences}
	\begin{align}
		\calV^0\; &\vcentcolon=\; \R^n, & \calV^{i+1} \;&\vcentcolon=\; \fA^{-1}(\fE\calV^i), &  i\in\N,\\
		\calW^0 \;&\vcentcolon= \;\{0\}, & \calW^{j+1} \;&\vcentcolon=\; \fE^{-1}(\fA\calW^j), & j\in\N,\label{eqn:WongSequencesImp}
	\end{align}
\end{subequations}
where we use the notation for the preimage as defined in~\eqref{eqn:pre:img}. After finitely many steps, the sequences in \eqref{eqn:WongSequences} converge, and the limits are given by
\begin{equation}
	\label{eqn:WongSequences:limits}
	\calV^\star \vcentcolon= \bigcap_{i\in\N} \calV^i \qquad\text{and}\qquad \calW^\star \vcentcolon= \bigcup_{i\in\N} \calW^i.
\end{equation}
	
\begin{theorem}[{\QWF via Wong sequences, \cite[Thm.~2.6]{BerIT12}}]
	\label{thm:QWF:Wong}  
	Consider a regular matrix pair $(\fE,\fA)$ with corresponding Wong limits $\calV^\star$ and $\calW^\star$. For any full rank matrices~$\difW$ and $\impW$ such that $\img(\difW)= \calV^\star$ and $\img(\impW) = \calW^\star$, the matrices 
	\begin{equation}\label{eqn:Wong:mat}
		\fT = [\difW, \impW],\quad \fS = [\fE\difW, \fA\impW]^{-1}
	\end{equation}
	are invertible and transform $(\fE,\fA)$ into \QWF~\eqref{eqn:QWF}.
\end{theorem}

With these preparations, we can aim for a geometric description of the solution of~\eqref{eqn:DAE} independent of the particular choice of coordinates for $\calV^\star$ and $\calW^\star$. Let us decompose the solution $\stx$ of~\eqref{eqn:DAE} as $\stx = \stx^{\diff} \oplus\stx^{\imp}$ with $\stx^{\diff}(t)\in\calV^\star$ and $\stx^{\imp}(t)\in \calW^\star$ for all $t\in\R$ and define the matrices
\begin{align}
	\label{eqn:projSel}
	\mathbf{\Pi}_{(\fE,\fA)} &\vcentcolon= \fT \begin{bmatrix}
		\fI & \zeroMat \\
		\zeroMat & \zeroMat
	\end{bmatrix} \fT^{-1}, &
	\mathbf{\Pi}^{\diff}_{(\fE,\fA)} &\vcentcolon= \fT \begin{bmatrix}
		\fI & \zeroMat \\
		\zeroMat & \zeroMat
	\end{bmatrix}\fS, &
	\mathbf{\Pi}^{\imp}_{(\fE,\fA)} &\vcentcolon= \fT \begin{bmatrix}
		\zeroMat & \zeroMat \\
		\zeroMat & \fI
	\end{bmatrix}\fS,
\end{align}
which are called the \emph{consistency projector}, the \emph{differential selector}, and the \emph{impulse selector}, respectively. Note that the projectors do not depend on the specific choice of $\fS,\fT\in\GL{\stateDim}$ (see \cite[Sec.~4.2.2]{Tre09-thesis}), and hence the matrices 
\begin{subequations}
	\label{eqn:diffImpMatrices}
	\begin{align}
		\fA^{\diff} &\vcentcolon= \mathbf{\Pi}^{\diff}_{(\fE, \fA)}\fA, & 
		\fB^{\diff} &\vcentcolon= \mathbf{\Pi}^{\diff}_{(\fE, \fA)}\fB, & 
		\fC^{\diff} &\vcentcolon= \fC \mathbf{\Pi}_{(\fE, \fA)},\label{eqn:diffImpMatricesA}\\
		\fE^{\imp} &\vcentcolon= \mathbf{\Pi}^{\imp}_{(\fE, \fA)}\fE, & 
		\fB^{\imp} &\vcentcolon= \mathbf{\Pi}^{\imp}_{(\fE, \fA)}\fB, &
		\fC^{\imp} &\vcentcolon= \fC \left( \fI-  \mathbf{\Pi}_{(\fE, \fA)}\right),
	\end{align}
\end{subequations}
are also independent of $\fS$ and $\fT$. 

For the \SDAE~\eqref{eqn:sDAE} we define accordingly $\mathbf{\Pi}_k \vcentcolon= \mathbf{\Pi}_{(\fE_k,\fA_k)}$, $\mathbf{\Pi}_k^{\diff} \vcentcolon= \mathbf{\Pi}^{\diff}_{(\fE_k,\fA_k)}$, $\mathbf{\Pi}_k^{\imp} \vcentcolon= \mathbf{\Pi}^{\imp}_{(\fE_k,\fA_k)}$, and the corresponding matrices $\fA_k^{\diff}$, $\fB_k^{\diff}$, $\fC_k^{\diff}$, $\fE_k^{\imp}$, $\fB_k^{\imp}$, and $\fC_k^{\imp}$ for $k\in\switchingSet$.

\begin{remark}
	Note that the matrices \eqref{eqn:diffImpMatrices} and \eqref{eqn:Wong:mat} should not be formed explicitly. In particular, in a large-scale context, the sparsity of the system matrices of the switched system may be destroyed by applying the operators in \eqref{eqn:projSel}. For this reason, assuming a sparse representation of \eqref{eqn:Wong:mat} (see the forthcoming discussion in \Cref{subsec:num:Wong}), the matrices \eqref{eqn:diffImpMatrices} should only be used implicitly through suitable matrix-vector products and sparse linear system solves. 
\end{remark}

\begin{remark}
	\label{rem:classicalSolution}
	For the existence of a continuously differentiable solution, the input~$\inp$ needs to be sufficiently smooth, and the initial value $\stx_0\in\R^{\stateDim}$ has to satisfy the \emph{consistency condition}
\begin{equation}
	\label{eqn:consistencyCondition}
		(\fI - \mathbf{\Pi}_{(\fE, \fA)})\stx_0 = \sum_{i=0}^{\indDAE-1}\left(\fE^{\imp}\right)^i\fB^{\imp}\inp^{(i)}(t_0).
\end{equation}
\end{remark}

%-----------------------------------------------------------------------------%	
\subsection{Switched descriptor systems}
\label{subsec:sDAE}

For the switched \DAE~\eqref{eqn:sDAE}, there is a priori no guarantee that the initial value after a switch satisfies the consistency condition~\eqref{eqn:consistencyCondition}. Hence, we cannot expect to assume the existence of classical solutions. Instead, we rely on \emph{piecewise-smooth distributions}, first presented in \cite{Tre09} and applied to switched \DAE in \cite{Tre09-thesis,Tre12}. In particular, regularity of each mode of the \SDAE and the assumption that the switching times do not accumulate is sufficient for the existence of a unique solution to the initial trajectory problem for~\eqref{eqn:sDAE}; see \cite[Cor.~5.2]{Tre12}. For an overview of other distributional solution concepts in the context of \DAEs, we refer to \cite{Tre13}. Asymptotic stability of \SDAEs 
is beyond the purposes of this work, although it is fundamental for practical applications. We refer to \cite{LibT12} for sufficient conditions for asymptotic stability of \SDAE under our class of switching signals defined in~\eqref{eqn:suitableSwitchingSignals} or under sufficiently slow average dwell-time switching. 

To formulate our forthcoming \MOR method, we make use of an equivalent characterization of the switched descriptor system as a switched \ODE with jumps and impulses, which was first derived in \cite{Hos22,HosT23}.
In more detail, let $\switch\in\calS$ and consider the switched \ODE with jumps and Dirac impulses given by
\begin{align}
	\label{eqn:switchedODEjump}
	\left\{\quad\begin{aligned}
		\dot{\fz}(t) &= \fA^{\diff}_{\switch_k}\fz(t) + \fB_{\switch_k}^{\diff}\inp(t), & t\in(t_k,t_{k+1}),\\
		\fz(t_k^{+}) &= \mathbf{\Pi}_{\switch_k}\fz(t_k^{-}) + \mathbf{\Pi}_{\switch_k}\fU_{\switch_k}^{-}(t_k^-), & \fz(t_0^{-}) = 0,\\
		\out(t) &= \fC^{\diff}_{\switch_k}\fz(t) + \fD_{\switch_k}\fU_{\switch_k}(t), & t\in(t_k,t_{k+1}),\\
		\out[t_k] &= -\fC^{\imp}_{\switch_k}\sum_{i=1}^{\indDAE_k-1}(\fE_{\switch_k}^{\imp})^{i}\fz(t_k^{-})\delta_{t_k}^{(i)} + \fU_{\switch_k}^{\imp}(t_k),
	\end{aligned}\right.
\end{align}
where $\switch_k\vcentcolon = \switch(t_k)$ is constant along $(t_k,t_{k+1})$, $\indDAE_{\switch_k}$ is the nilpotency index of matrix pencil $(\fE_{\switch_k},\fA_{\switch_k})$, and
\begin{subequations}
	\label{eqn:def:jum:imp}
	\begin{align}
		\fU_{\switch_k}^{-}(t_k^-) &\vcentcolon= \sum_{i=0}^{\indDAE_k-1}\left(\fE_{\switch_{k-1}}^{\imp}\right)^i\fB_{\switch_{k-1}}^{\imp}\inp^{(i)}(t_k^{-}),\\
		\fU_{\switch_k}^{\imp}(t_k) &\vcentcolon= \fC^{\imp}_{\switch_k}\sum_{i=1}^{\indDAE_k-1}\left(\fE_{\switch_k}^{\imp}\right)^{i}\left(  \fU_{\switch_{k+1}}^{-}(t_k^+)-\fU_{\switch_k}^-(t^-_k) \right)\delta_{t_k}^{(i)},\\
		\fU_{\switch_k}(t) &\vcentcolon= \left[\inp^\T(t),\dot{\inp}^{\T}(t),\ldots,\inp^{{(\indDAE_k-1)}^\T}(t)\right]^\T,\label{eqn:def:jum:imp:c}\\
		\fD_{\switch_k} &\vcentcolon= -\fC^{\imp}_{\switch_k}[\fE_{\switch_k}^{\imp}\fB_{\switch_k}^{\imp},\ldots,(\fE_{\switch_k}^{\imp})^{\indDAE_k-1}\fB_{\switch_k}^{\imp}].
	\end{align}
\end{subequations}
Note that we do not recall the precise definition of left-and right-sided evaluation at $t$, denoted with $t^-$ and $t^+$, the impulsive component denoted with $[t]$ here, and the $i$th derivative of the Dirac impulse $\delta_t^{(i)}$, since these aspects are not relevant in the remainder of the manuscript. Instead, we refer to \cite{Tre09}.
	
\begin{theorem}[{\!\cite[Thm.~7.11]{Hos22}}]
	\label{thm:equivReformulation}
	Assume that the switched \DAE~\eqref{eqn:sDAE} is regular. Then for every switching signal $\switch\in\calS$, the systems~\eqref{eqn:sDAE} and~\eqref{eqn:switchedODEjump} have the same input-output behavior in the space of piecewise-smooth distributions.
\end{theorem}
	
Let us emphasize that in contrast to a standard switched linear system, the switched system~\eqref{eqn:switchedODEjump} has the following additional features:
\begin{enumerate}
	\item The state transition in~\eqref{eqn:switchedODEjump} from one mode to another may depend on derivatives of the input such that additional input-dependent jumps may occur in the solution. 
	\item As a result of the impulsive component of the original switched descriptor system, derivatives of the input may appear as a direct feedthrough term in the output of~\eqref{eqn:switchedODEjump}.
	\item Inconsistent states before a switch may result in additional Dirac impulses in the output of~\eqref{eqn:switchedODEjump}.
\end{enumerate}

\subsection{Reachability and Observability for \SLS}
\label{subsec:reachObserv}
Let $\boldsymbol{\phi}( t, t_0, \stx_0, \inp, \switch)$ denote the state trajectory at time $t$ of a generic \SLS starting from $\stx(t_0) = \stx_0$ with input~$\inp$ and switching path $\switch\in\calS$ with $\calS$ given in \eqref{eqn:suitableSwitchingSignals}. 

\begin{definition}(see \cite[Sec.~4.2.1]{SunG05})\label{def1}
	Let $\switch\in\calS$ be a given switching path. A state $\stx\in\R^{\stateDim}$ is called
	\begin{enumerate}
		\item \emph{reachable via $\switch$} if there exists a time instant $t_{\mathrm{f}} > t_0$ and an input $\inp\colon [t_0,t_f] \rightarrow \R^{\inpDim}$ such that $\boldsymbol{\phi}(t_f,t_0,0,\inp,\switch) = \stx$;
		\item \emph{unobservable via $\switch$} if there exists an input $\inp$ such that
		\begin{equation*}
			\fC_{\switch}\boldsymbol{\phi}(t,t_0,\stx,\inp,\switch) = \fC_\switch \boldsymbol{\phi}(t, t_0,\zeroVec,\inp,\switch)\quad \text{for all } t \ge t_0.
		\end{equation*}
	\end{enumerate}
	The \emph{reachable} and \emph{unobservable sets} via $\switch$, denoted by $\calR_\switch$ and $\calU\calO_\switch$, respectively, are the sets of states that are reachable and unobservable via $\switch$, respectively. We define the \emph{observable set} via $\switch$ of \eqref{eqn:sDAE}, denoted by $\calO_\switch$, as $\calO_\switch \vcentcolon= (\calU \calO_\switch)^{\perp}$ (note that this definition is not unique since any complement set, not just the orthogonal one, would be suitable). The set of reachable states $\calR$ and the set of observable states $\calO$ of \eqref{eqn:sDAE} can be defined as
	\begin{align}
		\label{eqn:reachableObservableSet}
		\calR \vcentcolon= \bigcup_{\switch\in\calS}\calR_{\switch} \qquad\text{and}\qquad
		\calO \vcentcolon= \bigcup_{\switch\in\calS}\calO_{\switch}.
	\end{align}
\end{definition}
In the case of linear time-invariant systems in control form (without switching, jumps, and impulses), i.e.,
\begin{align}
	\label{eqn:LTI:system}
	\left\{\quad\begin{aligned}
		\dot{\state}(t) &= \fA\state(t)+{\fB}{\inp}(t), &  \state(t_0^{-}) &= \zeroVec,\\
		\out(t) &= \fC \state(t),
	\end{aligned}\right.
\end{align}
the reachable and observable sets are rigorously characterized. One way to do so is via $\fA$-invariant subspaces. Indeed, it can be shown that the reachable and observable sets of \eqref{eqn:LTI:system} are given, respectively, by $ \langle \fA \mid\img(\fB)\rangle$ and the orthogonal complement of $\langle \ker(\fC)\mid\fA \rangle$. See, for instance, \cite[Lem.~2.3]{SunG05} for the reachable set and \cite[Thm.~2.3.1]{Dai89} for the observable set. 

Reachability and observability are crucial in balancing-based \MOR. Indeed, the fundamental idea of this type of projection-based \MOR is to remove the states that are difficult to reach and/or difficult to observe; see \cite[Cha.~$7$]{Ant05} for further details.

%---------------------------------------------
The next lemma catheterizes the reachable and observable set via $\switch$ of a \SLS with (only) state dependent jumps through the matrices $\mathbf{\Pi}_{\switch_{k}}$. 
\begin{lemma}[See Lem.~4.5 in \cite{Hos22}]\label{lemma1}
	For a given switching signal $\switch\in\calS$ with $K\in\N_0$ switches at time instants $t_k$ with $k=0,\ldots,K$, define $\tau_k \vcentcolon= t_{k+1}-t_k$, and consider the following recursive relation 
	\begin{align*}
		\calM_0 &\vcentcolon= \calR_{\switch_0}, & 
		\calM_k &\vcentcolon= \calR_{\switch_k} + \ee^{\fA_{\switch_k}^{\diff}\tau_k}\mathbf{\Pi}_{\switch_k}\calM_{k-1}, & k &=1,\ldots,K,\\
		\calN_\kappa &\vcentcolon= \calU\calO_{\switch_\kappa}, &
		\calN_k &\vcentcolon= \calU\calO_{\switch_k}\cap\left(\ee^{-\fA_{\switch_k}^{\diff}\tau_k}\mathbf{\Pi}^{-1}_{\switch_{k+1}}\left(\calN_{\switch_{k+1}}\right)\right), & k &= K-1,\ldots, 0,
	\end{align*}
	where $\calR_{\switch_k} \vcentcolon= \big\langle \fA^{\diff}_{\switch_k} \mid \img(\fB^{\diff}_{\switch_k})\big\rangle$, $\calU\calO_{\switch_{k}} \vcentcolon= \big\langle  \Kern (\fC^{\diff}_{\switch_{k}}) \mid \fA^{\diff}_{\switch_{k}} \big\rangle$, and $\mathbf{\Pi}^{-1}_{\switch_{k+1}}\left(\calN_{\switch_{k+1}}\right)$ is the preimage of $\mathbf{\Pi}_{\switch_{k+1}}$ on the set $\calN_{\switch_{k+1}}$.
	Then the reachable and observable set via $\switch$ satisfy
	\begin{equation*}
		\calR_{\switch} = \calM_\kappa,\qquad \text{and}\qquad 
		\calO_{\switch}=\left(  \calU\calO_{\switch}    \right)^{\perp}\;=\; \calN_0^{\perp};
	\end{equation*}
	respectively.
\end{lemma}

\subsection{The piecewise balancing reduction for \SLS}\label{subsec:PBR-SLS}
Consider a \SLS of the form 
\begin{equation}
	\label{eqn:SLS}
	\Sigma_{\SLS} \quad \left\{\quad \begin{aligned}
		 \dot{\stx}(t)  &= \fA_{\switch(t)} \stx(t) +\fB_{\switch(t)}\inp(t), & \stx(t_0) &= \zeroVec, \\
		\out(t) &= \fC_{\switch(t)}\stx(t),\\
	\end{aligned}\right.
\end{equation}
with $\fA_j$ asymptotically stable for all $j=1,\ldots,M$. We start by defining the following $M$ pairs of \GLEs
\begin{subequations}
	\label{eqn:GLE:b}
	\begin{align}
		\fA_i \fcP_i+\fcP_i\fA_i^\T+\sum_{j=1}^{M}\left(\fN_{i,j} \fcP_i\fN_{i,j}^\T+\fB_j\fB_j^\T \right) &= \zeroVec,\label{eqn:GLE:reach:b}\\
		\fA_i^\T \fcQ_i + \fcQ_i\fA_i+\sum_{j=1}^{M}\left(\fN^\T_{i,j} \fcQ_i\fN_{i,j} + \fC_j^\T\fC_j\right) &= \zeroVec;\label{eqn:GLE:observ:b}
	\end{align}
\end{subequations}
	where $\fN_{i,j}\vcentcolon=\fA_j-\fA_i$ for $i=1,\ldots,M$. The symmetric positive semidefinite solutions  $\fcP_i$ and $\fcQ_i$ are known as the Gramians for the \SLS, and as shown in \cite[Thm.~3]{PonGB20}, they satisfy the relations 
	\begin{equation*}
		\calR =\text{range}(\fcP_i)\quad\text{and}\quad\calO=\text{range}(\fcQ_i)\quad\text{for }i=1,\ldots,M;
	\end{equation*}
	with $\calR$ and $\calO$, the reachable and observable set of the \SLS; see \Cref{def1}. Now, let  $\fcP_i =\fS_i\fS_i^\T$ and $\fcQ_i=\fR_i\fR_i^\T$ denote the Cholesky factors of the $i$th Gramian pair; compute the \emph{singular value decomposition} (\SVD) of the product of the Cholesky factors
	\begin{equation*}
		\fS_i^\T\fR_i = \left[\fU_{i,1}, \fU_{i,2}\right]\begin{bmatrix}
			\mathbf{\Sigma}_{i,1}&\zeroVec\\
			\zeroVec&\mathbf{\Sigma}_{i,2}
		\end{bmatrix}[\fV_{i,1},\fV_{i,2}]^\T,
	\end{equation*}
	and compute the projection matrices~$\fV_i$ and $\fW_i$ via
	\begin{equation*}
		\fV_i = \fS_i\fU_{i,1}\mathbf{\Sigma}^{-1/2}_{i,1} \quad\text{and}\quad
		\fW_i = \fR_i\fV_{i,1}\mathbf{\Sigma}^{-1/2}_{i,1}.
	\end{equation*}
	Then, we introduce the following reduced systems with state transition matrices
	\begin{equation}\label{eqn:new:red:system}
		 \Sigma_{\PBR-\SLS} \quad \left\{\quad \begin{aligned}
			\dot{\stateRed}(t)  &= \reduceBis{\fA}_{\switch(t)} \stateRed(t) + \reduceBis{\fB}_{\switch(t)}\inp(t), &	\stateRed(t_0) &= \zeroVec, \\
			\stateRed(t^{+}_{k})&=\fW^{\T}_{\switch_k}\fV_{\switch_{k-1}}\stateRed(t^-_k),\\
			\reduce{\out}(t) &= \reduceBis{\fC}_{\switch(t)}\stateRed(t),\\
		\end{aligned}\right.
	\end{equation}
	where $t_k$ are the switching times for the particular $\switch\in\calS$, $\switch_k\vcentcolon=\switch(t_k)$, and 
	\begin{align*}
		\reduceBis{\fA}_j &\vcentcolon= \fW_j^\T \fA_j \fV_j \in \R^{\stateDimRed_j\times \stateDimRed_j}, &
		\reduceBis{\fB}_j &\vcentcolon= \fW_j^\T \fB_j\in \R^{\stateDimRed_j\times \inpDim}, &
		\reduceBis{\fC}_j &\vcentcolon= \fC_j \fV_j\in\R^{\outDim\times \stateDimRed_j}
	\end{align*}
	for $j\in\switchingSet$. The procedure used to derive \eqref{eqn:new:red:system} from \eqref{eqn:SLS} is referred to as the \emph{piecewise balancing reduction} (\PBR) for \SLS. For the stability properties of the \PBR for \SLS, we refer to the discussion in \cite[Sec.~4.2]{ManU26}. Concerning accuracy, we now recall the results on the a-posteriori error estimate.
	
	\subsubsection{Accuracy of the \PBR for \SLS}\label{subsubsec:accuracy:PBR:SLS}
	For a switching signal $\switch\in\calS$, let us introduce the set of switching times up to time $t$, which we denote as $\calT_{\switch(t)}\vcentcolon=\{t_1,\ldots,t_K\}$, with $K\in\N_0$. Now, let us suppose that each subsystem $j$ of \eqref{eqn:SLS} is balanced via $ \fcP_i$ and $ \fcQ_i$; thus, the generalized observability and controllability Gramians are of the form
	\begin{equation}\label{eqn:bal:gra}
		\fSigma_j=\begin{bmatrix}
			\fSigma_{j,1}&\zeroVec\\
			\zeroVec&\sigma(j)\fI_{m(j)}
		\end{bmatrix} \quad\text{with}\quad\fSigma_{j,1}\in\R^{(\stateDim-m(j))\times(\stateDim-m(j))}
	\end{equation}
	where $m(j)$ is the multiplicity of $\sigma(j)$ implying that $\sigma(j)\not\in\fSigma_{j,1}$, being $\fSigma_{j,1}$ diagonal as a consequence of balancing. For simplicity, let us assume $\tilde m=m(j)$ for all $j\in\switchingSet$ (similar results can be stated without considering the minimum value of the multiplicity among the $M$ system modes). Rewriting the system matrices $\fA_j$, $\fB_j$, and $\fC_j$ for each $j\in\switchingSet$ in their balanced form we get
	\begin{equation}\label{eqn:SLS:balanced}
		\bar{\fA}_{j}\vcentcolon=\begin{bmatrix}
			\reduceBis{\fA}_{j}& \fA_{j,12}\\
			\fA_{j,21}& \fA_{j,22}
		\end{bmatrix},\quad  \bar{\fB}_{j}\vcentcolon=\begin{bmatrix}
			\reduceBis{\fB}_{j}\\
			\fB_{j,2}
		\end{bmatrix},\quad  \bar{\fC}_{j}\vcentcolon=\begin{bmatrix}
			\reduceBis{\fC}_{j}&
			\fC_{j,2}
		\end{bmatrix},
	\end{equation}
	where $\reduceBis{\fA}_{j}, \reduceBis{\fB}_{j}$, and $\reduceBis{\fC}_{j}$ is the $j$ subsystem of  \eqref{eqn:new:red:system} with $\stateDimRed=\stateDim-\tilde{m}$. Now, 	consider the \ROM~\eqref{eqn:new:red:system} for $\stateDimRed=\stateDim-\tilde m$ and the \FOM in balanced form \eqref{eqn:SLS:balanced}. Take the decomposition $\stateBal(t) = [\stateBal_1(t)^\T, \stateBal_2(t)^\T]^\T$ where $\stateBal_1(t) \in \R^{\stateDimRed}$ and let us define the vectors
	\begin{equation}\label{eqn:new:var}
		\bfX_c(t)\vcentcolon=\begin{bmatrix}
			\stateBal_1(t)+\stateRed(t)\\
			\stateBal_2(t)
		\end{bmatrix}\quad\text{and}\quad\bfX_o(t)\vcentcolon=\begin{bmatrix}
			\stateBal_1(t)-\stateRed(t)\\
			\stateBal_2(t)
		\end{bmatrix},
	\end{equation}
	and the quantities
	\begin{align}\label{eqn:swi:con:fac}
		\begin{aligned}
			\gleG_{o,\switch}(\stateDim,\stateDimRed)\;\vcentcolon=&\;{\bfX}_o^{\T}(t)\fSigma_{\switch(t)}{\bfX}_o(t) +\sum_{k=1}^{K}	\left({\bfX}_o^{\T}(t^{-}_k)\fSigma_{\switch(t^{-}_k)}{\bfX}_o(t^{-}_k)-{\bfX}_o^{\T}(t^{+}_k)\fSigma_{\switch(t^{+}_k)}{\bfX}_o(t^{+}_k)    \right),\\
			\gleG_{c,\switch}(\stateDim,\stateDimRed)\;\vcentcolon=&\;	\sigma(\switch(t_K^{+}))^2{\bfX}_c^{\T}(t)\fSigma^{-1}_{\switch(t)}{\bfX}_c(t) \\
			+&\;\sum_{k=1}^{K}	\left(\sigma(\switch(t^{-}_k))^2{\bfX}_c^{\T}(t^{-}_k)\fSigma^{-1}_{\switch(t^{-}_k)}{\bfX}_c(t^{-}_k)-\sigma(\switch(t_k^{+}))^2{\bfX}_c^{\T}(t^{+}_k)\fSigma^{-1}_{\switch(t^{+}_k)}{\bfX}_c(t^{+}_k)    \right),\\
			\fM_c(j)\;\vcentcolon=&\;\bar \fA_j\Sigma_j+\Sigma_j\bar \fA_j^{\T}+\bar \fB_j\bar\fB^{\T}_j,\quad	\fM_o(j)\;\vcentcolon=\;\bar \fA^{\T}_j\Sigma_j+\Sigma_j\bar \fA_j+\bar \fC_j^{\T}\bar\fC_j,\\
			\gleH_{c,\switch}(\stateDim,\stateDimRed)\;\vcentcolon=&\;\sum_{k=0}^{K}\lambda_{1}(	\fM_c(\switch(t_k^{+})))\int_{t_k}^{t_{k+1}}\|\bfX_c(s)\|_2^2\,\ds,\\
			\gleH_{o,\switch}(\stateDim,\stateDimRed)\;\vcentcolon=&\;\sum_{k=0}^{K}\lambda_{1}(	\fM_c(\switch(t_k^{+})))\int_{t_k}^{t_{k+1}}\|\bfX_o(s)\|_2^2\,\ds.
		\end{aligned}
	\end{align}

With this notation in place, we are nearly prepared to formulate the result on the error bound, namely \cite[Thm.~4.3]{ManU26}. For $\stateDimRed\ll\stateDim$ and $k\in\N$ with $k>\stateDimRed$, we denote by $\sigma_k(j)$ the distinct singular values of the balanced Gramian \eqref{eqn:bal:gra} and by $m_k(j)$ their corresponding multiplicities for $k= \stateDimRed+1, \ldots ,\tilde \stateDim$, respectively. To streamline the notation, we assume that $m_k(j_1)= m_k(j_2)$ holds for all $j_1,j_2\in\switchingSet$ and for all $k= \stateDimRed+1, \ldots ,\tilde \stateDim$. Under this assumption, we can omit the dependence on $j$ in $m_k(j)$ and thus avoid introducing additional cumbersome notation in the statement of the next theorem.
	\begin{theorem}[Error bound of the \PBR for \SLS; see \cite{ManU26}]\label{teo6}
		Given a switching signal $\switch\in\calS$, consider the \ROM~\eqref{eqn:new:red:system} of size $\stateDimRed$. Then, the output error between the \FOM~\eqref{eqn:sDAE} and the \ROM~\eqref{eqn:new:red:system} of size $\stateDimRed$ satisfies the relation 
		\begin{align}
			\label{eq:err:est:3}
			\begin{aligned}
				&\left(\int_{0}^{t}\|\out(s)-\tilde{\out}(s)\|^2_{2}\,\ds\right)^{\frac{1}{2}}
				\;\leq\;2\left(\int_{0}^{t}\|\inp(s)\|^2_{2}\,\ds\right)^{\frac{1}{2}}\sum_{k=\stateDimRed+1}^{\tilde \stateDim}\tilde{\sigma}_{s_{k}}\\
				+&
				\sum_{k=\stateDimRed+1}^{\tilde \stateDim}   \gleH_{c,o,\switch}(s_{k+1},s_{k})^{\frac{1}{2}}H(   \gleH_{c,o,\switch}(s_{k+1},s_{k}))\\
				-&
				\sum_{k=\stateDimRed+1}^{\tilde \stateDim}   \gleG_{c,o,\switch}(s_{k+1},s_{k})^{\frac{1}{2}}H(-   \gleG_{c,o,\switch}(s_{k+1},s_{k}))\;=\vcentcolon\;\tau(\stateDimRed,\inp),
			\end{aligned}
		\end{align}
		where $H(x)$ is the Heaviside function, $\tilde{\sigma}_k\vcentcolon=\max_{j\in\switchingSet}\sigma_{k}(j) $, $s_k\vcentcolon=\sum_{k=\stateDimRed}^{\tilde n}m_k$, $m_\stateDimRed\vcentcolon=\stateDimRed$, and
		\begin{align*}
			\begin{aligned}
				\gleH_{c,o,\switch}(s_{k+1},s_{k})	\;&\vcentcolon=\;\gleH_{o,\switch}(s_{k+1},s_{k})+\gleH_{c,\switch}(s_{k+1},s_{k}),\\
				\gleG_{c,o,\switch}(s_{k+1},s_{k})	\;&\vcentcolon=\;\gleG_{o,\switch}(s_{k+1},s_{k})+\gleG_{c,\switch}(s_{k+1},s_{k}).
			\end{aligned}
		\end{align*}
	\end{theorem}

	The error bound presented in \Cref{teo6} comprises three distinct contributions: the first corresponds to the sum of the truncated singular values, as in the classical a-priori error bound for balanced truncation; the second stems from the use of piecewise constant Gramians in the derivation of \eqref{eqn:new:red:system}; and the third reflects the fact that the numerically computed solutions of the \GLEs may not exactly satisfy certain \emph{linear matrix inequalities} (\LMIs) conditions due to approximation errors. We refer to \cite[Sec.~4.5]{ManU26} for a discussion on the computational aspects of \Cref{teo6}.
	
%-----------------------------------------------------------------------------%
\section{\MOR for \SDAE}
\label{sec:MORforDAE}
In this section, we demonstrate how the \PBR for \SLS introduced in \Cref{subsec:PBR-SLS} can likewise be applied to the \SDAE setting to obtain an accurate \ROM. In particular, in \Cref{Sec:NEW-FORMULATION}, we show that the input-output behavior of \eqref{eqn:sDAE} coincides with that of a \SLS whose state dimension varies over time and exhibits state-input-dependent jumps and impulsive outputs at the switching instants. In \Cref{eqn:ext:reach} and \Cref{eqn:ext:obs}, we examine the reachability and observability sets of this special class of \SLS and connect them to those of a \SLS with augmented input and output matrices. Finally, in \Cref{sec:PBR-sDAE}, we present the \PBR tailored to \SDAE.

\subsection{Reformulation of the switched system}\label{Sec:NEW-FORMULATION}
Let us observe that, in general, $\fA_{\switch_k}^{\diff}$ is singular because, in the switched system \eqref{eqn:switchedODEjump}, the state variable is represented in the full dimension $\stateDim$, while the dynamics are described solely through the differential part of each \DAE~subsystem, which, in general, evolves in a smaller dimension space. In what follows, we derive an alternative formulation of \eqref{eqn:switchedODEjump} that guarantees asymptotically stable matrices for each subsystem by employing an appropriate number of degrees of freedom for the state variable in every system mode. By employing the definitions \eqref{eqn:QWF}, \eqref{eqn:projSel}, and \eqref{eqn:diffImpMatrices} we can rewrite system \eqref{eqn:switchedODEjump} as
\begin{align}\label{eq:58}
	\left\{\quad\begin{aligned}
		\dot{\SwDifState}_{\switch_k}(t) &= \begin{bmatrix}
			\fJ_{\switch_k} & \zeroMat\\
			\zeroMat & \zeroMat
		\end{bmatrix}\SwDifState_{\switch_k}(t) + \begin{bmatrix}
			\fI_{\stateDim_{\fJ_{{\switch_k}}}}&\mathbf{0}\\
			\zeroMat & \zeroMat
		\end{bmatrix}\fS_{\switch_k}\fB_{\switch_k}\inp(t),	& t\in(t_k,t_{k+1}),\\
		\SwDifState_{\switch_k}(t_k^+) &=\begin{bmatrix}
			\fI_{\stateDim_{\fJ_{\switch_k}}} & \zeroMat\\
			\zeroMat & \zeroMat
		\end{bmatrix}	\StateTran_{\switch_{k},\switch_{k-1}}\SwDifState_{{\switch_{k-1}}}(t_k^{-})+	\fcB_{\switch_k,\switch_{k-1}}\fU_{\switch_{k-1}}(t^{-}_k), & \SwDifState(t_0^{-}) = \zeroVec,\\
		\out(t) &= \fC_{\switch_k} \fT_{\switch_k}   \begin{bmatrix}
			\fI_{\stateDim_{\fJ_{{\switch_k}}}} & \zeroMat\\
			\zeroMat & \zeroMat
		\end{bmatrix}       \SwDifState_{\switch_k}(t) + \fD_{\switch_k}\fU_{\switch_k}(t), & t\in(t_k,t_{k+1}),\\
		\out[t_k] &=-	\fcC_{\switch_{k},\switch_{k-1}}\begin{bmatrix}
			\mathbf{1}^{\T}_{\stateDim}\delta_{t_k}&\ldots&	\mathbf{1}^{\T}_{\stateDim}\delta_{t_k}^{\nu_{k}-1}
		\end{bmatrix}^{\T}{\SwDifState}_{\switch_{k-1}}(t^{-}_k)+ \fU_{\switch_k}^{\imp}(t_k).
	\end{aligned}\right.
\end{align}
with $\fz(t)=\fT_{\switch_k}\SwDifState_{\switch_k}(t)$, the notation convention $\nu_k=\nu_{\switch_k}$, and
\begin{align*}
	\fcB_{\switch_k,\switch_{k-1}}\vcentcolon=&\begin{bmatrix}
		\fI_{\stateDim_{\fJ_{\switch_{k}}}} & \zeroMat\\
		\zeroMat&\zeroMat
	\end{bmatrix} \fT_{\switch_{k}}^{-1}
	\begin{bmatrix}
		\begin{bmatrix}
			\mathbf{0}&\mathbf{0}\\
				\mathbf{0}&\fN_{\switch_{k-1}}^0
		\end{bmatrix}\fS_{\switch_{k-1}}\fB_{\switch_{k-1}}, \ldots, 	\begin{bmatrix}
		\mathbf{0}&\mathbf{0}\\
		\mathbf{0}&\fN_{\switch_{k-1}}^{\nu_{k-1}-1}
		\end{bmatrix}\fS_{\switch_{k-1}}\fB_{\switch_{k-1}}
	\end{bmatrix},\\
		\fcC_{\switch_{k},\switch_{k-1}} \vcentcolon=& \fC_{\switch_k}\fT_{\switch_{k}}\begin{bmatrix}
	  \begin{bmatrix}
		\mathbf{0}& \mathbf{0}\\
			\mathbf{0}&\fN_{\switch_k}
		\end{bmatrix}\StateTran_{\switch_{k},\switch_{k-1}},\ldots, \begin{bmatrix}
	    \mathbf{0}& \mathbf{0}\\
		\mathbf{0}&\fN_{\switch_k}^{\nu_{k}-1} 
		\end{bmatrix}\StateTran_{\switch_{k},\switch_{k-1}}
		\end{bmatrix},\\
	\StateTran_{\switch_{k},\switch_{k-1}}\vcentcolon=& \fT_{{\switch_k}}^{-1}\fT_{{\switch_{k-1}}}.
\end{align*}
Since \eqref{eq:58} is obtained via a state-space transformation, the input-output mapping is not altered. We immediately observe that the last $\stateDim-\stateDim_{\fJ_{\switch_k}}$ entries of $\SwDifState_{\switch_k}(t)$ are always zero; therefore, by defining the new variable $\SwDifStateRed_{\switch_k}(t)=\begin{bmatrix}
			\fI_{\stateDim_{\fJ_{{\switch_k}}}} & \zeroMat
		\end{bmatrix}\SwDifState_{\switch_k}(t)$, system \eqref{eq:58} can be reduced to 
\begin{align}\label{eq:59}
	\left\{\quad\begin{aligned}
		\dot{\SwDifStateRed}_{\switch_k}(t) &= \fJ_{\switch_k}\SwDifStateRed_{\switch_k}(t) + \begin{bmatrix}
			\fI_{\stateDim_{\fJ_{\switch_k}}} & \zeroMat
		\end{bmatrix}\fS_{\switch_k}\fB_{\switch_k}\inp(t), & t\in(t_k,t_{k+1}),\\
		\SwDifStateRed_{\switch_k}(t_k^+) &=\tilde	\StateTran_{{\switch_k},{\switch_{k-1}}} \SwDifStateRed_{\switch_{k-1}}(t_k^{-})+\tilde	\fcB_{\switch_k,\switch_{k-1}}\fU_{\switch_{k-1}}(t^{-}_k), & \SwDifStateRed(t_0^{-}) = \zeroVec,\\
		\out(t) &= \fC_{\switch_k} \fT_{\switch_k}   \begin{bmatrix}
			\fI_{\stateDim_{\fJ_{{\switch_k}}}} & \zeroMat
		\end{bmatrix}^\T  \SwDifStateRed_{\switch_k}(t) + \fD_{\switch_k}\fU_{\switch_k}(t), & t\in(t_k,t_{k+1}),\\
		\out[t_k] &=-	\tilde\fcC_{\switch_{k},\switch_{k-1}}\begin{bmatrix}
			\mathbf{1}^{\T}_{{\stateDim_{\fJ_{{\switch_{k-1}}}}}}\delta_{t_k}&\ldots&	\mathbf{1}^{\T}_{{\stateDim_{\fJ_{{\switch_{k-1}}}}}}\delta_{t_k}^{\nu_{k}-1}
		\end{bmatrix}^{\T}{\SwDifStateRed}_{\switch_{k-1}}(t^{-}_k)+ \fU_{\switch_k}^{\imp}(t_k)
	\end{aligned}\right.
\end{align}
with
\begin{align}\label{eqn:reformulation:SDAE}
\begin{aligned}
\tilde	\StateTran_{{\switch_k},{\switch_{k-1}}} \vcentcolon=& \begin{bmatrix}
		\fI_{\stateDim_{\fJ_{{\switch_k}}}}&\mathbf{0}
	\end{bmatrix}\fT_{{\switch_k}}^{-1}\fT_{{\switch_{k-1}}}\begin{bmatrix}
		\fI_{\stateDim_{\fJ_{{\switch_{k-1}}}}}&\mathbf{0}
	\end{bmatrix}^\T\;\in\;\R^{\stateDim_{\fJ_{{\switch_k}}}\times\stateDim_{\fJ_{\switch_{k-1}}}},\\
	\tilde	\fcB_{\switch_k,\switch_{k-1}}\vcentcolon=&\begin{bmatrix}
		\fI_{\stateDim_{\fJ_{\switch_{k}}}} & \zeroMat
	\end{bmatrix} \fT_{\switch_{k}}^{-1}
	\begin{bmatrix}
		\begin{bmatrix}
			\mathbf{0}&\mathbf{0}\\
			\mathbf{0}&\fN_{\switch_{k-1}}^0
		\end{bmatrix}\fS_{\switch_{k-1}}\fB_{\switch_{k-1}}, \ldots, 	\begin{bmatrix}
			\mathbf{0}&\mathbf{0}\\
			\mathbf{0}&\fN_{\switch_{k-1}}^{\nu_{k-1}-1}
		\end{bmatrix}\fS_{\switch_{k-1}}\fB_{\switch_{k-1}}
	\end{bmatrix},\\
		\tilde\fcC_{\switch_{k},\switch_{k-1}} \vcentcolon=& \fC_{\switch_k}\fT_{\switch_{k}}\begin{bmatrix}
		\begin{bmatrix}
			\mathbf{0}& \mathbf{0}\\
			\mathbf{0}&\fN_{\switch_k}
		\end{bmatrix}\StateTran_{\switch_{k},\switch_{k-1}}\begin{bmatrix}
		\fI_{\stateDim_{\fJ_{{\switch_{k-1}}}}}\\
        \mathbf{0}
		\end{bmatrix},\ldots, \begin{bmatrix}
			\mathbf{0}& \mathbf{0}\\
			\mathbf{0}&\fN_{\switch_k}^{\nu_{k}-1} 
		\end{bmatrix}\StateTran_{\switch_{k},\switch_{k-1}}\begin{bmatrix}
		\fI_{\stateDim_{\fJ_{{\switch_{k-1}}}}}\\ \mathbf{0}
		\end{bmatrix}
	\end{bmatrix}.
    \end{aligned}
\end{align}
Note that $\fJ_j$ is an asymptotically stable matrix for all $j=1,\ldots,M$ (see \Cref{rmk:AS}). In the next subsections, we characterize the reachability and observability sets of \eqref{eq:59} by relating them to those of a suitably defined \SLS with state jumps.

\subsection{Reachability for switched systems with input-state-dependent jumps}\label{eqn:ext:reach}
Consider a fixed switching signal $\switch\in\switchingSet$. To characterize the reachable set via $\switch$ of system~\eqref{eq:59}, we can restrict the analysis to the following input-to-state switched linear system with input-state-dependent jumps
\begin{align}\label{eq21}
	\left\{\quad\begin{aligned}
		\dot{\SwDifStateRed}_{\switch_k}(t) &= \fJ_{\switch_k}\SwDifStateRed_{\switch_k}(t) +\tilde\fB_{\switch_k}\inp(t), & t\in(t_k,t_{k+1}),\\
	\SwDifStateRed_{\switch_k}(t_k^+) &=\tilde	\StateTran_{{\switch_k},{\switch_{k-1}}} \SwDifStateRed_{\switch_{k-1}}(t_k^{-})+\tilde	\fcB_{\switch_k,\switch_{k-1}}\fU_{\switch_{k-1}}(t^{-}_k), & \SwDifStateRed(t_0^{-}) = \zeroVec, 
	\end{aligned}\right.
\end{align}
where $\tilde\fB_{\switch_k}\vcentcolon=\begin{bmatrix}
	\fI_{\stateDim_{\fJ_{\switch_k}}} & \zeroMat
\end{bmatrix}\fS_{\switch_k}\fB_{\switch_k}$. The solution of~\eqref{eq21} for $t \in [t_k,t_{k+1})$ and $k\in\{0,\ldots,K-1\}$ is given recursively by
\begin{align}\label{eq25}
	\begin{aligned}
\SwDifStateRed_{\switch_k}(t) \;=&\; \ee^{\fJ_{\switch_k}(t-t_k)}\left(\tilde	\StateTran_{{\switch_k},{\switch_{k-1}}} \SwDifStateRed_{\switch_{k-1}}(t_k^{-})+\tilde	\fcB_{\switch_k,\switch_{k-1}}\fU_{\switch_{k-1}}(t^{-}_k),  \right)\\
+&\;\int_{t_k}^{t}\ee^{\fJ_{\switch_k}(t-s)}\tilde\fB_{\switch_k}\inp(s)\,\ds.
	\end{aligned}
\end{align}
Let us now introduce the following definition of reachable set via $\switch$ for a given time interval.
\begin{definition}\label{def:rech:ex:reach}
	For a given switching signal $\switch\in\calS$, the reachable and modified reachable subspace of the switched system \eqref{eq21} on the time interval $[t_0,t)$ are defined, respectively, by
	\begin{align*}
		\calR_{\switch}(t_0,t) &\vcentcolon= \{\SwDifStateRed_{\switch(t^-)} (t^-) \mid \exists \text{ a solution } (\SwDifStateRed_\switch,\inp) \text{ of \eqref{eq21} in $[t_0,t)$ with } \fz(t_0^-)=0\},\\
		\tilde \calR_{\switch}(t_0,t) &\vcentcolon= \left\{\SwDifStateRedBis_{\switch(t^-)}(t^-) \,\left|\,\begin{aligned}
			& \SwDifStateRedBis_{\switch(t^-)}(t^-) =\tilde	\StateTran_{\switch(t^+),\switch(t^-)}\SwDifStateRed_{\switch(t^-)} (t^-)+\tilde	\fcB_{\switch(t^+),\switch(t^-)}\fU_{\switch(t^-)}(t^{-})\\
			&\text{ for }\SwDifStateRed_{\switch(t^-)} (t^-)\in	\calR_{\switch}(t_0,t) \text{ and }\fU_{\switch(t^-)}(t^-)\text{ as in \eqref{eqn:def:jum:imp:c}}\\ 
			& \text{ with } \inp \text{ such that }(\SwDifStateRed_{\switch} (t) ,\inp(t))\;\text{for }t\in[t_0,t),\\
			& \text{ is a solution of }\eqref{eq21} \text{ for } \fz(t_0^-)=0
		\end{aligned}\right.\right\}.
	\end{align*}
\end{definition}

	By definition, we have $\calR_{\switch}(t_0,t)\subseteq \calR_\switch$ and $\bigcup_{t>t_0}\calR_{\switch}(t_0,t)=\calR_\switch$, with $\calR_{\switch}$ as in \Cref{def1}. 
	
Let us also formally introduce the local reachable subspace for the mode $k$, i.e., $\calR_{k}$, and its modified version, $\tilde \calR_{k}$, which corresponds to the one in \Cref{def:rech:ex:reach} when the switching signal is constantly equal to $k$ in the prescribed time frame, i.e., $\switch(t)=k$ in $[t_0,t)$. In the following, we will make use of \cite[Lem.~2.3]{SunG05} which states equivalence of $\calR_{k}$ with the smallest invariant subspace containing $\img(\tilde \fB_{k})$, i.e., we have $\calR_{k} \equiv \langle \fJ_{\switch_k} \mid \img(\tilde \fB_{k})\rangle$.

With these preparations, we obtain the following generalization of \Cref{lemma1} where only state-dependent jumps were considered.

\begin{lemma}\label{lemma:reachability}
For a given switching signal $\switch\in\calS$ with $K\in\N_0$ switches at times $t_k$ with $k=1,\ldots, K$, consider the recursive relations
	\begin{align}\label{eq23}
		\begin{aligned}
			\tilde{\calM}_0 &\vcentcolon= \tilde{\calR}_{\switch_0}, &
			\tilde{\calM}_k &\vcentcolon= \tilde{\calR}_{\switch_k}+\ee^{\fJ_{\switch_k}\tau_k}\tilde \calM_{k-1} ,\quad k= 1,\ldots, K,\\
			\calM_0 &\vcentcolon= \calR_{\switch_0}, & 
			\calM_k &\vcentcolon= \calR_{\switch_k}+\ee^{\fJ_{\switch_k}\tau_k}\tilde \calM_{k-1} ,\quad k= 1,\ldots, K,\\
		\end{aligned}
	\end{align}
    where $\tilde{\calR}_{\switch_0}$ and $\calR_{\switch_0}$ denote the reachable and modified reachable set for the active system mode at time $t_0$ and $\tau_k \vcentcolon= t_{k+1}-t_k$ is the duration of the mode $\switch_k$. Then, the reachable and modified reachable set via $\switch$, in the time interval $[t_0,\tfinal)$, with $\tfinal=t_{K+1}$ at any time $\tfinal>t_K$, for~\eqref{eq21} are given by
	\begin{equation*}
		\calR_{\switch}(t_0,\tfinal) = \calM_K, \quad \tilde \calR_{\switch}(t_0,\tfinal) = \tilde \calM_K.
	\end{equation*}
\end{lemma}

\begin{proof}
	The proof is based on the induction principle. By definition $\calM_0 = \calR_{\switch_0} = \calR_\switch (t_0,t_1)$ and $\tilde \calM_0 = \tilde \calR_{\switch_0} = \tilde \calR_\switch (t_0,t_1)$. Further, assume 
	\begin{equation}\label{eq24}
		\calR_\switch(t_0,t_k)=\calM_{k-1},\quad \tilde \calR_\switch(t_0,t_k)=\tilde \calM_{k-1}
	\end{equation}
	for some $k\le K$. Let $\fz_{k+1}\in \calM_k$, i.e., there exist $\SwDifStateRedBis_{k}\in\tilde \calM_{k-1}$, and $\fz_{\inp}\in\calR_{\switch_k}$ such that 
	\begin{equation}\label{eqn:zkplus1}
		\fz_{k+1}\in \calM_k,\quad\quad \fz_{k+1} = \ee^{\fJ_{\switch_{k}}\tau_k}\SwDifStateRedBis_k+\fz_{\inp}.
	\end{equation}
	From \eqref{eq24}, using the definition of modified reachability via $\switch$, it follows that there exists a control ${\inp_1}$ defined in the time interval $[t_0,t_k)$ such that the associated solution $\SwDifStateRed_1$ of~\eqref{eq21} in this interval satisfies ${\fz_1}(t_0^-) = \zeroVec$ and $\tilde\StateTran_{\switch(t_k^+),\switch(t_k^-)}\SwDifStateRed_1(t_k^-)=\SwDifStateRedBis_1(t_k^-)-\tilde	\fcB_{\switch(t_k^+),\switch(t_k^-)}\hat{\fU}_{1,\switch_{k-1}}(t^{-}_k)$. Now let us extend $(\SwDifStateRed_1 ,{\inp_1})$ to the time interval $[t_0,t_{k+1})$ via 
	\begin{equation*}
		(\SwDifStateRed_1(t),{\inp_1}(t))\;\vcentcolon=\;\left(  \ee^{\fJ_{\switch_k}(t-t_k)}\left(\tilde \StateTran_{\switch(t_k^+),\switch(t_k^-)}\SwDifStateRed_1(t_k^-)+\tilde	\fcB_{\switch(t_k^+),\switch(t_k^-)}\hat{\fU}_{1,\switch_{k-1}}(t^{-}_k) \right),\zeroVec\right)
	\end{equation*}
	for $t\in[t_k,t_{k+1})$. Using~\eqref{eq25}, we immediately obtain that $(\SwDifStateRed_1,\inp_1)$ is a solution of~\eqref{eq21} on $[t_0,t_{k+1})$. Moreover, there exists a solution $(\SwDifStateRed_2,\inp_2)$ of mode $\switch_k$ on $(t_k,t_{k+1})$ with $\SwDifStateRed_2(t_k^+)=\zeroVec$ and $ \SwDifStateRed_2(t_{k+1}^-)=\fz_{\inp}$. Further, we set $\left(\SwDifStateRed_2(t), {\inp_2}(t) \right)=(\zeroVec,\zeroVec)$ for all $t\in[t_0,t_k]$ and observe that $\left(\SwDifStateRed_2, \inp_2 \right)$ is a solution of the switched system \eqref{eq21} on $[t_0,t_{k+1})$ with $\fz(t_0^-)=\zeroVec$. Note that here it is crucial that ${\inp_2}(t)=\zeroVec$ for all $t\in[t_0,\;t_k)$ because this ensures ${\fU}_{2,\switch_{k-1}}(t^-_k)=\zeroVec$.
	By linearity $(\SwDifStateRed,\inp) \vcentcolon= (\SwDifStateRed_1 + \SwDifStateRed_2,{\inp_1} + \inp_2)$ is still a solution of \eqref{eq21} in $[t_0,t_{k+1}) $ with $\fz(t_0^-)=\zeroVec$ and
	\begin{align*}
		\SwDifStateRed(t_{k+1}^-) \;= &\;\SwDifStateRed_1(t_{k+1}^-)+\SwDifStateRed_2(t_{k+1}^-)\\
		 =&\; \ee^{\fJ_{\switch_k}\tau_k}\left(\tilde \StateTran_{\switch(t_k^+),\switch(t_k^-)}\SwDifStateRed_1(t_k^-)+\tilde	\fcB_{\switch(t_k^+),\switch(t_k^-)}\hat{\fU}_{1,\switch_{k-1}}(t^{-}_k) \right)+\fz_{\inp} \\
		=&\;\ee^{\fJ_{\switch_k}\tau_k}	\SwDifStateRedBis_k+\fz_{\inp} \;=\; \fz_{k+1},
	\end{align*}
	which together with \eqref{eqn:zkplus1} implies $\fz_{k+1}\in\calR_\switch(t_0,t_{k+1})$ and hence $\calM_{k}\subseteq \calR_\switch(t_0,t_{k+1})$. Considering $\tilde \calM_{k}$ and thus $\fz_{\inp}\in\tilde\calR_{\switch_k}$, repeating the same argument, we obtain $\tilde \calM_{k}\subseteq \tilde \calR_\switch(t_0,t_{k+1})$.
	
	To show the reverse inclusion, let $\fz_{k+1}\in\calR_\switch(t_0,t_{k+1})$, i.e., there exists a control input $\inp$ such that the associated solution $(\SwDifStateRed,\inp)$ of~\eqref{eq21} satisfies $\SwDifStateRed(t^{-}_{k+1}) = \fz_{k+1}$. From $\SwDifStateRed(t_k^-)\in\calR_\switch(t_0,t_{k})=\calM_{k-1}$ and
	\begin{equation*}
		\fz_{\inp}\;\vcentcolon=\;\int_{t_k}^{t_{k+1}}\ee^{\fJ_{\switch_k}(t_{k+1}-t)}\tilde \fB_{\switch_k}\inp(t)\;\dt\in\calR_{\switch_k},
	\end{equation*}
	it follows immediately from \eqref{eq25} and the definition of the modified reachable set that 
	\begin{align*}
		\fz_{k+1} = \SwDifStateRed({t^{-}_{k+1}}) &= \ee^{\fJ_{\switch_k}\tau_k}\left(\tilde\StateTran_{\switch(t_k^+),\switch(t_k^-)}\SwDifStateRed(t_k^-)+\tilde	\fcB_{\switch(t_k^+),\switch(t_k^-)}\hat{\fU}_{1,\switch_{k-1}}(t^{-}_k)\right)+\fz_{\inp}\\
		&\in \ee^{\fJ_{\switch_k}\tau_k}\tilde \calM_{k-1}    +\calR_{\switch_k} = \calM_{k}.
	\end{align*}
	The proof for $\tilde \calR_\switch(t_0,t_{k+1})\subseteq \tilde \calM_k$ follows analogously by observing that the input at time $t^{-}_{k+1}$ appears in $\fz_{\inp}$; therefore, it is enough to consider $\fz_{\inp}\in \tilde \calR_{\switch_k}$ as prescribed by the definition of $\tilde \calM_k$.
\end{proof}

The next theorem relates the reachable set of a system with input-state-dependent jumps to the reachable set of a system where the jumps are only state-dependent but with augmented input matrices.

\begin{theorem}
	\label{teo:reachability}
	The reachable set via $\switch\in\calS$ of~\eqref{eq21} is contained within the reachable set via $\switch$ of the system
	\begin{align}\label{eq26}
		\left\{\quad\begin{aligned}
			\dot{\SwDifStateRed}_{\switch_k}(t) &=\fJ_{\switch_k}\SwDifStateRed_{\switch_k}(t) + \bar{\fB}_{\switch_k,\switch_{k-1}} \bar{\inp}_{\switch_{k}}(t), &\qquad t\in(t_k,t_{k+1}),\\
			\SwDifStateRed_{\switch_k}(t_k^+) &=\tilde \StateTran_{\switch_k,\switch_{k-1}}\SwDifStateRed_{\switch_{k-1}}(t_k^{-}), & \SwDifStateRed(t_0^{-}) = \zeroVec,
		\end{aligned}\right.
	\end{align}
	where
	\begin{align}
		\label{eq1}
		\bar{\fB}_{\switch_k,\switch_{k-1}} &\vcentcolon= \begin{bmatrix}
			\tilde\fB_{\switch_k} &\tilde	\fcB_{\switch_k,\switch_{k-1}}
		\end{bmatrix}, &
		\bar{\inp}_{\switch_k}(t) &\in \R^{\inpDim+\inpDim\nu_{\switch_{k-1}}}.
	\end{align}	
\end{theorem}
\begin{proof}
	For a fixed signal $\switch\in\calS$, the reachable set via $\switch$ of~\eqref{eq21} is given by \Cref{lemma1} and the observation after \Cref{def:rech:ex:reach}. Using \cite[Lem.~2.3]{SunG05}, the smallest $\fJ_{\switch_k}$-invariant subspace that contains $\img\left(\tilde	\fcB_{\switch(t_k^+),\switch(t_k^-)}\right)$, i.e. 
	\begin{align}\label{eq:20}
		\calG_{\switch_k,\switch_{k-1}}\;=&\; \left \langle \fJ_{\switch_k}\;\Big|\;\img\left(\tilde	\fcB_{\switch_k,\switch_{k-1}}\right)\right\rangle,
	\end{align}
	is equal to the subspace 
	\begin{align}\label{eq20}
		\begin{aligned}
			\calW_{\switch_k,\switch_{k-1}} &\vcentcolon= \spann\left\{\ee^{ \fJ_{\switch_k}t}\tilde	\fcB_{\switch_k,\switch_{k-1}}\fa \mid t\in [0,\tau_{k}], \fa\in\R^{\inpDim\nu_{\switch_{k-1}}}\right\}\\
			&\phantom{\vcentcolon}= \spann\left\{\ee^{ \fJ_{\switch_k}t}\fb \mid t\in [0,\tau_{k}],\fb\in 
			\img{\left(\tilde	\fcB_{\switch_k,\switch_{k-1}}\right)}\right\}.
		\end{aligned}
	\end{align}
	Using the equivalence of \eqref{eq:20} and \eqref{eq20} and the definition of the modified reachable set in a prescribed time interval, we can rewrite $\calM_k$ in \eqref{eq23} as 
	\begin{align}
		\begin{aligned}
			{\calM}_k \;=&\;\calR_{\switch_k}+\ee^{\fJ_{\switch_k}\tau_k}\tilde \calM_{k-1}\\ \subseteq&\;\calR_{\switch_k}+\ee^{\fJ_{\switch_k}\tau_k}\left(\tilde\StateTran_{{\switch_k},{\switch_{k-1}}}\calM_{k-1}+\img\left(\tilde	\fcB_{\switch_k,\switch_{k-1}}\right)\right)\\
			\subseteq&\; \calR_{\switch_k}  +\ee^{\fJ_{\switch_k}\tau_k}\tilde \StateTran_{\switch_k,\switch_{k-1}}\calM_{k-1}+  \left \langle \fJ_{\switch_k}\;\Big|\;\img\left(\tilde	\fcB_{\switch_k,\switch_{k-1}}\right)\right\rangle      \\
			=&\;   \left \langle \fJ_{\switch_k}\;\Big|\;\img(\tilde \fB_{\switch_k})+\img\left(\tilde	\fcB_{\switch_k,\switch_{k-1}} \right)\right\rangle  +\ee^{\fJ_{\switch_k}\tau_k}\tilde\StateTran_{\switch_k,\switch_{k-1}}\calM_{k-1} \\
			=&\;\left \langle \fJ_{\switch_k}\;\Big|\;\img\left(\bar{\fB}_{\switch_k,\switch_{k-1}}\right)\right\rangle + \ee^{\fJ_{\switch_k}\tau_k}\tilde\StateTran_{\switch_k,\switch_{k-1}}\calM_{k-1},
		\end{aligned}
	\end{align}
for all $k = 1,\ldots, K$, which coincides with the reachability set generated by $\switch$ over the time interval considered for the state-dependent jump system \eqref{eq26}; see \Cref{lemma1}.
\end{proof}

\subsection{Observability set for state-dependent impulses}\label{eqn:ext:obs}
We recall from~\eqref{eq:59} that the output function is given as
\begin{align}
	\label{eqn:out:jumps}
	\begin{aligned}
		\out(t) =& \tilde \fC_{\switch_k}  \SwDifStateRed_{\switch_k}(t) + \fD_{\switch_k}\fU_{\switch_k}(t), \quad t\in(t_k,t_{k+1}),\\
	\out[t_k] =&-	\tilde\fcC_{\switch_{k},\switch_{k-1}}\begin{bmatrix}
	\mathbf{1}^{\T}_{{\stateDim_{\fJ_{{\switch_{k-1}}}}}}\delta_{t_k}&\ldots&	\mathbf{1}^{\T}_{{\stateDim_{\fJ_{{\switch_{k-1}}}}}}\delta_{t_k}^{\nu_{k}-1}
\end{bmatrix}^{\T}{\SwDifStateRed}_{\switch_{k-1}}(t^{-}_k)+ \fU_{\switch_k}^{\imp}(t_k),
	\end{aligned}
\end{align}
where $\tilde \fC_j\vcentcolon= \fC_{j} \fT_{j}   \begin{bmatrix}
	\fI_{\stateDim_{\fJ_{j}}} & \zeroMat
\end{bmatrix}^\T  $ for $j=1,\ldots,M$. We define the observable set as the orthogonal complement of the unobservable set. We then formally specify the unobservable set for $\switch\in\calS$ over the time interval $(t, \tfinal)$ as follows.

\begin{definition}\label{def2}
	For a given switching signal $\switch\in\calS$, the unobservable subspace of the switched system~\eqref{eq:59} for this switching signal on the time interval $(t,\tfinal)$ is defined as
	\begin{align}
		\calU_{\switch}(t,\tfinal) \vcentcolon= \left\{ \SwDifStateRed_{\switch(t^+)}(t^+) \,\left| \, \begin{aligned}
		 & \SwDifStateRed_{\switch}\text{ is a solution of \eqref{eq:59} with } \inp \equiv \zeroVec \text{ such that } \out(s) = \zeroVec\\ 
		 &\text{ for all } s\in(t,\tfinal)\end{aligned}\right.\right\}
	\end{align}
\end{definition}

\begin{proposition}
	Let $\switch\in\calS$. Then $\calU_{\switch}(t_0,\infty) = \calU\calO_{\switch}$ where $ \calU\calO_{\switch}$ is as defined in \Cref{def1}.
\end{proposition}

\begin{proof}
	Let $\fz \in \calU_{\switch}(t_0,\infty)$. Then, by definition, we have $\out(t;\boldsymbol{\phi}(t,t_0,\fz,\zeroVec,\switch))=\zeroVec$ and $\out(t;\boldsymbol{\phi}(t,t_0,\zeroVec,\zeroVec,\switch))=\zeroVec$ since zero initial conditions and constant zero input lead to the null trajectory. Hence, $\fz\in\calU\calO_{\switch}$. Conversely, let $\fz\in \calU\calO_{\switch}$, i.e.,
	\begin{equation}\label{eqn:zero:out}
		\out(t;\boldsymbol{\phi}(t,t_0,\fz,\inp,q))-	\out(t;\boldsymbol{\phi}(t,t_0,\zeroVec,\inp,q))=\zeroVec.
	\end{equation}
	Linearity immediately implies $\out(t;\boldsymbol{\phi}(t,t_0,\fz,\zeroVec,q))=\zeroVec$ and hence $\fz\in\calU_{\switch}(t_0,\infty)$.
\end{proof}
With this groundwork in place, we can now present the following result, which describes the unobservable set of system \eqref{eq:59}. This result extends \cite[Lem.~4.7]{Hos22} (see also \cite{KueT16}), in which the impulsive response at the switching instants is not taken into account.
\begin{lemma}\label{lemma3}
	Let $\switch\in\calS$ with  $K\in \N$ switches at time points $t_k$ with $k=0,\ldots, K$, define $\tau_k \vcentcolon= t_{k+1}-t_k$, and consider the recursive relation 
	\begin{align}\label{eqn:rec:unob}
		\begin{aligned}
			{\calN}_K &\vcentcolon= \calU_{\switch_K},\\
			{\calN}_k &\vcentcolon= \calU_{\switch_k} \cap \left( \ee^{-\fJ_{\switch_k}\tau_k}\left(\left(\tilde	\StateTran_{{\switch_{k+1}},{\switch_{k}}}^{-1}\left(\calN_{k+1}\right)\right)\cap\ker(	\tilde\fcC_{\switch_{k+1},\switch_{k}})\right)\right),\quad k= K-1,\ldots, 0,\\
		\end{aligned}
	\end{align}
	where $\calU_{\switch_k}=\big\langle \ker(\tilde \fC_{\switch_k}) \mid \fJ_{\switch_k}\big \rangle$ and $\StateTran_{{\switch_{k+1}},{\switch_{k}}}^{-1}(\calN_{k+1})$ denotes the preimage defined in~\eqref{eqn:pre:img}.
	Then, the unobservable set via $\switch$ for~\eqref{eqn:switchedODEjump} is given by $\calN_k = \calU_{\switch}(t_k,\tfinal)$
	for any $\tfinal>t_K$ and for any $k=0,\ldots,K$.
\end{lemma}
\begin{proof}
	We proceed by induction. For $k=K$, we have $\calN_K = \calU_{\switch_k} = \calU_{\switch}(t_K,\tfinal)$. 
	Assume now $\calN_{k+1}=\calU_{\switch}(t_{k+1},\tfinal)$ for some $k \leq K-1$. Let $\fz_k\in\calN_k$, then $\fz_k\in\calU_{\switch_k}$ and there exists
	\begin{equation}\label{eq:1d}
		\fz_{k+1}\in\left(\left(\tilde	\StateTran_{{\switch_{k+1}},{\switch_{k}}}^{-1}\left(\calN_{k+1}\right)\right)\cap\ker(\tilde\fcC_{\switch_{k+1},\switch_{k}})\right)
	\end{equation}
	such that $\fz_{k+1}=\ee^{\fJ_{\switch_k}\tau_k}\fz_k$. Thus, there exists a solution $\fz$ of \eqref{eq:59}, with $\fu=\zeroVec$ on $[t_k,\tfinal)$ and $\fz(t_k^+)=$~$\fz_k$ satisfying $\out=\zeroVec$ on $(t_k,t_{k+1})$, since $\fz_k\in\calU_{\switch_k}$. Moreover, $\out=\zeroVec$ on $[t_{k+1}, \tfinal)$ since
	\[
	\fz(t^{-}_{k+1})=\fz_{k+1}\in\left(\left(\tilde	\StateTran_{{\switch_{k+1}},{\switch_{k}}}^{-1}\left(\calU_{\switch}(t_{k+1},\tfinal)\right)\right)\cap\ker\left(\tilde\fcC_{\switch_{k+1},\switch_{k}}\right)\right)
	\]  
	where we used the inductive assumption $\calN_{k+1}=\calU_{\switch}(t_{k+1},\tfinal)$ on \eqref{eq:1d}.
	This implies that $\fz_k\in\calU^{\switch}_{[t_k,\tfinal)}$.
	
	Conversely, let $\fz_k\in\calU^{\switch}_{[t_k,\tfinal)}$. Then, there exists a solution of \eqref{eq:59} in $[t_k,\tfinal)$ with constant input $\inp=\zeroVec$ and $\fz(t_k^+)=\fz_k$ that satisfies $\out(t)=\zeroVec$ for all $t\in[t_k,\tfinal)$. Clearly, because the solution is such that $\out=\zeroVec$ on $[t_k, t_{k+1})$, $\fz_k\in\calU_{\switch_k}$. Moreover, using this solution at time $t_{k+1}^{-}$ and the fact that $\out=\zeroVec$ on $[t_{k+1}, \tfinal)$, we can define
	\begin{align*}
    \begin{aligned}
	\fz_{k+1}\;\vcentcolon=&\;\fz(t^-_{k+1})\;\in\;\left(\left(\tilde	\StateTran_{{\switch_{k+1}},{\switch_{k}}}^{-1}\left(\calU_{\switch}(t_{k+1},\tfinal)\right)\right)\cap\ker(\tilde\fcC_{\switch_{k+1},\switch_{k}})\right)\\
    =
    &\;\left(\tilde	\StateTran_{{\switch_{k+1}},{\switch_{k}}}^{-1}\left(\calN_{{k+1}}\right)\right)\cap\ker(\tilde\fcC_{\switch_{k+1},\switch_{k}})
    \end{aligned}
	\end{align*}
	where we used the inductive principle assumption and $\fz_{k+1}\in\ker(\tilde\fcC_{\switch_{k+1},\switch_{k}})$ comes from the fact that $\out(t_{k+1})=\zeroVec$ which, by \eqref{eqn:out:jumps}, directly implies $\fz(t_{k+1}^-)\in\ker(\tilde\fcC_{\switch_{k+1},\switch_{k}})$. From $\fz_{k+1}=\ee^{\fJ_{\switch_k}\tau_k}\fz_k$, it follows that 
	\[
	\fz_k\in\ee^{-\fJ_{\switch_k}\tau_k}\{\fz_{k+1}\} \subseteq \ee^{-\fJ_{\switch_k}\tau_k}\left(\left(\tilde	\StateTran_{{\switch_{k+1}},{\switch_{k}}}^{-1}\left(\calN_{{k+1}}\right)\right)\;\cap\;\ker(\tilde\fcC_{\switch_{k+1},\switch_{k}})\right).
	\]
	Then, by the fact that $\fz_k\in\calU_{\switch_k}$ and definition \eqref{eqn:rec:unob}, we can conclude that $\fz_k\in\calN_k$.
\end{proof}

\begin{remark}
	Since the input function is considered to be zero throughout the time interval, cf.~\Cref{def2}, the input-dependent part of the jumps in system \eqref{eq:59} does not play a role in defining the unobservable set.
\end{remark}

Analogously to the reachable set, we now derive a description of the unobservable set of system \eqref{eq:59} by relating it to the unobservable set of an appropriately defined switched system whose output function does not include Dirac delta impulses at the switching instants.
\begin{theorem}\label{teo2bis}	
	The unobservable set via the fixed switched signal $\switch\in\calS$ of system \eqref{eq:59} equals the unobservable set of the system
	\begin{align}\label{eq26:tris}
		\left\{\quad\begin{aligned}
			\dot{\SwDifStateRed}_{\switch_k}(t) &= \fJ_{\switch_k}\SwDifStateRed_{\switch_k}(t), &\quad t\in(t_k,t_{k+1}),\\
			\SwDifStateRed_{\switch_k}(t_k^+) &= \tilde\StateTran_{{\switch_k},{\switch_{k-1}}} \SwDifStateRed_{\switch_{k-1}}(t_k^{-})  &\SwDifStateRed(t_0^{-})=\fz_0,\\
			{\out}_{\switch_k}(t) &= \bar{\fC}_{\switch_{k+1},\switch_{k}}\SwDifStateRed_{\switch_k}(t)
		\end{aligned}\right.
	\end{align}
	where
	\begin{align}\label{eq1:bis}
		\bar{\fC}_{\switch_{k+1},\switch_{k}} &\vcentcolon= \begin{bmatrix}
			\tilde\fC_{\switch_k}^{\T}& 	-\fcC^\T_{\switch_{k+1},\switch_{k}} 
		\end{bmatrix}^\T, &
		{\out}_{\switch_k}(t) &\in \R^{\outDim\nu_{\switch_{k}}}.
	\end{align}	
\end{theorem}

\begin{proof}
	Observe that~\eqref{eq26:tris} does not depend on any input. Indeed, in the study of the observable set, the input function can be assumed to be constantly zero. For a fixed switching signal~$\switch\in\calS$, the observable set of system \eqref{eq:59} is given by \Cref{lemma3}. Using a classical result, see for instance \cite[Thm.~2.3.1]{Dai89}, we can state that the largest $\fJ_{\switch_k}$-invariant subspace contained in $\ker(\fcC_{\switch_{k+1},\switch_{k}})$, i.e. 
	\begin{align}\label{eq20:}
		\calT_{\switch_{k+1},\switch_{k}} &= \left \langle \ker(\fcC_{\switch_{k+1},\switch_{k}})\;\Big|\;\fJ_{\switch_k} \right\rangle,
	\end{align}
	is equal to the subspace 
	\begin{align}\label{eq20:bis}
		\begin{aligned}
			\calW_{\switch_{k+1},\switch_{k}}\;\vcentcolon=&\;\spann\left\{\ee^{-\fJ_{\switch_k}t}\fz : t\in [0,\tau_k],\fz\in\ker(\fcC_{\switch_{k+1},\switch_{k}})\right\}.
		\end{aligned}
	\end{align}
	Using \eqref{eq20:}-\eqref{eq20:bis}, we can rewrite $\calN_k$ in \eqref{eqn:rec:unob} as 
	\begin{align}
		\begin{aligned}
			{\calN}_k \;=&\; \calU_{\switch_k}\;\cap\;\ee^{-\fJ_{\switch_k}\tau_k}\left(\left(\tilde	\StateTran_{{\switch_{k+1}},{\switch_{k}}}^{-1}\left(\calN_{k+1}\right)\right)\cap\ker(\fcC_{\switch_{k+1},\switch_k})\right)\\
			\subseteq&\; \calU_{\switch_k}  \;\cap\;\ee^{-\fJ_{\switch_k}\tau_k}\tilde\StateTran_{{\switch_{k+1}},{\switch_{k}}}^{-1}\left(\calN_{k+1}\right)\;\cap\;  \left \langle\ker(\fcC_{\switch_{k+1},\switch_k}) \;\mid\; \fJ_{\switch_k}\right\rangle      \\
			=&\; \left \langle\ker(\tilde\fC_{\switch_k})\;\cap\;\ker(\fcC_{\switch_{k+1},\switch_{k}}) \;\mid\; \fJ_{\switch_k}\right\rangle \; \cap\;\ee^{-\fJ_{\switch_k}\tau_k}\tilde\StateTran_{{\switch_{k+1}},{\switch_{k}}}^{-1}\left(\calN_{k+1}\right) \\
			=&\; \left \langle\ker(\bar{\fC}_{\switch_k}) \;\mid\; \fJ_{\switch_k}\right\rangle\;\cap\; \ee^{-\fJ_{\switch_k}\tau_k}\tilde\StateTran_{{\switch_{k+1}},{\switch_{k}}}^{-1}\left(\calN_{k+1}\right),
		\end{aligned}
	\end{align}
	for all $k=1,\ldots, K$, which coincides with the unobservable set of system \eqref{eq26:tris}; see \cite[Lem.~4.7]{Hos22}. Observing that we can choose any duration $\tau_k>0$ for a mode $\switch_k$ we can conclude that the sets are equal.
\end{proof}
%----------------------------------------------

\subsection{The \PBR for \SDAE}
\label{sec:PBR-sDAE} 
Let us introduce the following switched system with state-dependent jumps
\begin{align}
	\label{eqn:switchedODEjump:tris}
	\left\{\quad\begin{aligned}
		\dot{\SwDifStateRed}_{\switch_k}(t) &= \fJ_{\switch_k}\SwDifStateRed_{\switch_k}(t) +	\bar{\fB}_{\switch_k,\switch_{k-1}} \bar{\inp}_{\switch_k}(t), & t\in(t_k,t_{k+1}),\\
		\SwDifStateRed_{\switch_k}(t_k^{+}) &=  \tilde\StateTran_{{\switch_k},{\switch_{k-1}}} \SwDifStateRed_{\switch_{k-1}}(t_k^{-}), & \SwDifStateRed(t_0^{-}) = 0,\\
		{\out}_{\switch_k}(t) &= \bar{\fC}_{\switch_{k+1},\switch_{k}}\SwDifStateRed_{\switch_k}(t)
	\end{aligned}\right.
\end{align}
where $\bar{\fB}_{\switch_k,\switch_{k-1}}$ and $\bar{\fC}_{\switch_{k+1},\switch_{k}}$ are as in \eqref{eq1} and \eqref{eq1:bis}, respectively. The observability set of \eqref{eqn:switchedODEjump:tris} is the same as that of \eqref{eq:59} by \Cref{teo2bis}; similarly, the reachability set also coincides, provided that, for the components of the augmented input defined in \eqref{eq26}, we allow impulsive inputs, since this corresponds to incorporating an associated non-zero initial term into the initial data; see, e.g., \cite[Lem.~B5]{Hos22}. Thus, since the input-output map of \eqref{eq:59} can be recovered from that of \eqref{eqn:switchedODEjump:tris}, after a suitable definition of the input function, we aim to apply the \PBR method presented in \Cref{subsec:PBR-SLS} to \eqref{eqn:switchedODEjump:tris}. With this in mind, we first solve the $M$ pairs of \GLEs of the following form
\begin{subequations}
	\label{eqn:GLE:c}
	\begin{align}
		\fJ_i \fcP_i+\fcP_i\fJ_i^\T+\sum_{j=1}^{M}\fN_{i,j} \fcP_i\fN_{i,j}^\T+\sum_{j=1}^{M}\sum_{k=1}^{M}\fB_{j,k}\fB_{j,k}^\T  &= \zeroVec,\label{eqn:GLE:reach:c}\\
		\fJ_i^\T \fcQ_i + \fcQ_i\fJ_i+\sum_{j=1}^{M} \fN^\T_{i,j} \fcQ_i\fN_{i,j} + \sum_{j=1}^{M}\sum_{k=1}^{M} \fC_{j,k}^\T \fC_{j,k} &= \zeroVec,\label{eqn:GLE:observ:c}
	\end{align}
\end{subequations}
where $\fN_{i,j}=\tilde\StateTran_{i,j} \fJ_{j}\tilde\StateTran_{j,i} -\fJ_{i}$, $ \fB_{j,k}=\tilde\StateTran_{i,j}\bar \fB_{j,k}$, and $ \fC_{j,k}=\bar \fC_{j,k}\tilde\StateTran_{k,i}$ .

Following the procedure detailed in \Cref{subsec:PBR-SLS} we obtain the projection matrices $\fV_i, \fW_i $ from the truncated \SVD of the product of the Cholesky factors of each matrix pair $(\fcP_i,\fcQ_i)$. Assuming we truncate at $\stateDimRed$ for all $i=1,\ldots,M$ we end-up with the \ROM 
\begin{equation}\label{eqn:new:red:system:b}
	\switchedSysRedJumps \quad \left\{\quad \begin{aligned}
		\dot{\SwDifStateRedBis}(t)  &= \reduceBis{\fJ}_{\switch_k} \SwDifStateRedBis(t) + \reduceBis{\fB}_{\switch_k,\switch_{k-1}}\bar \inp_{\switch_k}(t), &	\SwDifStateRedBis(t_0) &= \zeroVec, \\
		\SwDifStateRedBis(t^{+}_{k})&=	\hat \fPi_{\switch_k,\switch_{k-1}}\SwDifStateRedBis(t^-_k),\\
		\bar{\out}_{\switch_k}(t) &= \reduceBis{\fC}_{\switch_{k+1},\switch_{k}}\SwDifStateRedBis(t),\\
	\end{aligned}\right.
\end{equation}
where $t_k$ are the switching times for the particular $\switch\in\calS$, $\switch_k\vcentcolon=\switch(t_k)$, and 
\begin{align*}
	\begin{aligned}
	\reduceBis{\fJ}_j\; &\vcentcolon=\; \fW_j^\T \fJ_j \fV_j \in \R^{\stateDimRed \times \stateDimRed}, \qquad
	\reduceBis{\fB}_{j,i} \;\vcentcolon=\; \fW_j^\T \bar \fB_{j,i}\in \R^{\stateDimRed\times( \inpDim+\inpDim\indDAE_i)}, \\
	\hat \fPi_{j,i}\;&\vcentcolon=\;\fW^{\T}_{j}\tilde \StateTran_{j,i}\fV_{i}\in\R^{\stateDimRed\times\stateDimRed},\quad\;\;\; \reduceBis{\fC}_{j,i} \;\vcentcolon=\;  \bar \fC_{j,i}\fV_i \in\R^{\outDim\indDAE_j\times \stateDimRed},
	\end{aligned}
\end{align*}
for $i,j\in\switchingSet$; we call this procedure the \PBR for \SDAEs.  
To evaluate the accuracy of this reduction procedure, by following \Cref{subsubsec:accuracy:PBR:SLS}, we can directly apply \Cref{teo6} and thus have an evaluable a-posteriori error bound for the reduction error. 

\begin{remark}
    Note that, with respect to the \PBR for \SLS presented in \cite[Sec.~4]{ManU26}, the \SLS in \eqref{eqn:switchedODEjump:tris} exhibits state-dependent jumps at the switching times. Nevertheless, the \PBR naturally generalizes to \SLS with state-dependent jumps by suitably defining the $\fN_{i,j}$ matrices as in \eqref{eqn:GLE:c} and incorporating the corresponding state jumps into the \ROM \eqref{eqn:new:red:system:b}. Since this extension is straightforward, we omit the proof for brevity and avoid repeating the arguments of \cite[App.~B]{ManU24} in the presence of state-dependent jumps.

\end{remark}

\begin{remark}\label{rmk-l2-linf}
Using \Cref{teo6}, we have a bound for the error
	\begin{equation*}
		\left(\int_{0}^{t}\|\out_{\switch_k}(s)-\bar{\out}_{\switch_k}(s)\|^2_{2}\,\ds\right)^{\frac{1}{2}},
	\end{equation*}
	being $\out_{\switch_k}$ and $\bar{\out}_{\switch_k}$ the output functions of system \eqref{eqn:switchedODEjump:tris} and \eqref{eqn:new:red:system:b}, respectively. Looking at system \eqref{eq:59}, recalling that we denote with $\calT_{\switch(t)}\vcentcolon=\{t_1,\ldots,t_K\}$ and $K\in\N$ the set of switching times up to time $t$, it appears more natural to bound the following error
	\begin{equation*}
		\left(\int_{0}^{t}\|\out(s)-\bar{\out}(s)\|^2_{2}\,\ds+\sum_{k=1}^{K}\|\out(t_k)-\bar{\out}(t_k)\|^{2}\right)^{\frac{1}{2}},
	\end{equation*}
	with $\bar{\out}$ the output of a suitably constructed \ROM for system \eqref{eq:59}.
	This, for a finite time interval $[0,t]$, could be achieved by bounding the $L_{\infty}$-norm of the output reduction error over that interval, i.e.
		\begin{equation*}
\max_{s\in[0,t]}\|\out_{\switch_k}(s)-\bar{\out}_{\switch_k}(s)\|_{2}.
	\end{equation*}
	A balancing based reduction procedure is, by construction, tailored for $L_2$-norm output error control and does not appear straightforward, at the moment, to have this procedure controlling the $L_{\infty}$-norm of the output error. Other types of reduction procedures, like the ones based on rational interpolations \cite{AntBG20}, are designed to control the $L_{\infty}$-norm of the output error; however, such types of reduction procedures, in the context of \SLS (and thus also \SDAE), are, to the best of our knowledge, still to be investigated.
\end{remark}

%-----------------------------------------------------------------------------%
\section{Numerical considerations on the \PBR for \SDAE}
\label{sec:numericalDetails}
In large-scale scenarios, it is essential to take advantage of sparsity and to employ state-of-the-art techniques at every stage of the \MOR procedure described in \Cref{sec:MORforDAE}. The approach involves two main computational steps:
\begin{enumerate}
\item decoupling the differential and impulsive parts of the \SDAE in order to obtain either implicit or explicit access to \eqref{eq:59};
\item computing the $M$ pairs of projection matrices $(\fV_i,\fW_i)$, which entails solving the \GLE in \eqref{eqn:GLE:c}.
\end{enumerate}
We address the first aspect in \Cref{subsec:num:Wong}, whereas in \Cref{eqn:app:GLE} we discuss the approximation of the \GLE solution.

%-----------------------------------------------------------------------------%
\subsection{Discussion on the computational efficiency of the decoupling}
\label{subsec:num:Wong}
The first numerical issue concerns the determination of the two subspaces $\calV^\star$ and $\calW^\star$ from~\eqref{eqn:WongSequences:limits} and two representative matrices, $\difW$ and $\impW$ such that $\image{(\difW)}=\calV^\star$ and $\image{(\impW)}=\calW^\star$, necessary to transform the \DAE~\eqref{eqn:DAE} in the \QWF; see \Cref{teo:QWF}. To be efficient, an algorithm determining $\calV^\star$ and $\calW^\star$ for large systems, should retain the following two features:

\begin{enumerate}
	\item if the matrix pair $(\fE,\fA) $ is sparse, then forming dense matrices of size $\stateDim\times\stateDim$ to determine $\difW$, $\impW$, and  the matrices in \eqref{eqn:diffImpMatrices} should be avoided;\label{item:1}
	\item the computation of $\difW$ and $\impW$ should be efficient, i.e., involving at the most $\calO(\stateDim)$ floating point operations for sparse matrices and $\calO(\stateDim^2)$ for full matrices.\label{item:2}
\end{enumerate}
In general, a decoupling strategy that is able to satisfy both \ref{item:1} and \ref{item:2} for a general class of problems does not appear to be available. Practically, working only with sparse matrices by resorting to implicit definitions of the matrices is possible (thus satisfying \ref{item:1}). The problem relies on the efficiency, in the sense detailed in \ref{item:2}, and also accounts for the extra floating point operations necessary when the matrices are implicitly defined.

Let us observe that, by \eqref{eqn:WongSequences} and the observation that the sequence converges after $\indDAE$ steps, see \cite[Prop.~2.10]{BerIT12}  (where $\indDAE$ is the nilpotent index of the \DAE), the kernel of $\fE$ has a crucial role; indeed, from \eqref{eqn:WongSequences}, one can easily check that 
\begin{equation}
 \calW^1=\ker(\fE),\quad \calV^1=\fA^{-1}\fE \ker(\fE)^\perp.
\end{equation}
We emphasize that in many applications, see for instance \cite[Sec.~5]{MehU23}, $\indDAE$ is typically known and is less than or equal to three, i.e., the limiting spaces~$\calV^\star$ and~$\calW^\star$ in~\eqref{eqn:WongSequences:limits} are determined after a few iterations. Moreover, the systems arise in a highly structured form that either simplifies the computation or can be used to directly read off the desired matrices~$\difW$ and~$\impW$ in a sparse manner.  Therefore, it is often the case that a sparse matrix representation with orthogonal columns of $\ker(\fE)$, say $\fK$, can be straightforwardly derived. For instance, this is the case when 
\begin{equation}\label{eq:60}
	\fE=\begin{bmatrix}
		\fM & \zeroMat\\
		\zeroMat & \zeroMat
	\end{bmatrix}
\end{equation}
with $\fM$ a square invertible matrix.
\subsection{Approximation of the \GLE solution}\label{eqn:app:GLE}
Let us define the generic \GLE of the form
\begin{equation}
	\label{eqn:generic:GLE:2}
	\fA \fX+\fX\fA^\T+\sum_{j=1}^{M}\left(\fN_j \fX\fN_j^\T\right)+\fB\fB^{\T} = \zeroVec,
\end{equation}
where $\fA,\fN_j\in\R^{\stateDim\times\stateDim}$, $\fA$ is Hurwitz and $\fB\in\R^{\stateDim\times\inpDim}$. The existence and uniqueness of a solution of~\eqref{eqn:generic:GLE:2} can be studied via Kronecker algebra, i.e., by introducing the matrices
\begin{align}\label{eqn:GLE:oper}
	\lyapOper \vcentcolon= \fI_\stateDim\otimes\fA +\fA\otimes\fI_\stateDim \in\R^{\stateDim^2\times\stateDim^2} \qquad\text{and}\qquad
	\gleOper \vcentcolon= \sum_{j=1}^{M}\fN_j\otimes\fN_j \in\R^{\stateDim^2\times\stateDim^2},
\end{align}
such that the vectorized form of~\eqref{eqn:generic:GLE:2} is given by $(\lyapOper+\gleOper)\vec2(\fX)=-\vec2(\fB\fB^\T)$.
Hence, the \GLE~\eqref{eqn:generic:GLE:2} is uniquely solvable if and only if $\lyapOper+\gleOper$ is nonsingular. 

Since we aim to avoid solving a linear system of dimension $\stateDim^2$ and, consequently, storing a solution vector of dimension $\stateDim^2$, we employ the methodology presented in \cite[Sec.~3]{ManU26} to approximate the solution of \eqref{eqn:generic:GLE:2} with error control. Such a procedure is based on the use of the stationary iteration algorithm from \cite[Alg.~2.1]{ShaSS16} and a suitably derived stopping criterion. For completeness, we present the stationary iteration algorithm in \Cref{alg:statIterGLE}. It consists of fixed-point iterations in which we need to solve a classical Lyapunov equation in each iteration. To guarantee formal convergence, we make the following assumption and refer to \cite[Rem.~3.2]{ManU26} for the discussion on its general validity.
\begin{assumption}
	\label{ass:statIterGLE}
It holds that $\|\lyapOper^{-1}\gleOper\|_2<1$.
\end{assumption}

\begin{algorithm}[ht]
	\caption{Stationary iterations for the \GLEs~\eqref{eqn:generic:GLE:2}}
	\label{alg:statIterGLE}
	\begin{minipage}{\linewidth}
		\textbf{Input:} Matrices $\fA$, $\fB$,  $\fN_j$ for $j=1,\ldots,M$\\
		\textbf{Output:} $\tilde{\fZ}$ such that $\fX\approx\tilde{\fZ}\tilde{\fZ}^\T$ is an approximation to the solution of \eqref{eqn:generic:GLE:2}
	\end{minipage}
	\begin{algorithmic}[1]
		\State Set ${\fB}_1 \vcentcolon=\fB$
		\State Find $\tilde \fX_1=\tilde \fZ_1\tilde \fZ_1^\T$ approximately solving $\fA\fX_1+\fX_1\fA^\T+{\fB}_1{\fB}_1^\T=\zeroMat$, for $\fX_1=\fZ_1\fZ_1^\T$\label{line2}
		\For{$k=2,3,\ldots$}
		\State Set ${\fB}_k \vcentcolon= [\fN_1\fZ_{k-1},\ldots,\fN_M\fZ_{k-1},{\fB}_1]$
		\State Find $\tilde \fX_k=\tilde \fZ_k\tilde \fZ_k^\T$ that approximately solve \label{A2-l5}
		\begin{equation}
			\label{alg:statIterGLE:lyap}
			\fA\fX_k+\fX_k\fA^\T+{\fB}_k{\fB}^\T_k = \zeroMat
		\end{equation}
		\phantom{\textbf{for}}for $\fX_k=\fZ_k\fZ_k^\T$
		\State \textbf{if} {sufficiently accurate} \textbf{then} stop
		\EndFor 
	\end{algorithmic}
\end{algorithm}
We denote the approximate solution of \eqref{alg:statIterGLE:lyap}, at iteration $k$ of \Cref{alg:statIterGLE}, by $\tilde{\fX}_k$ and observe that $\tilde{\fX}_k=\tilde{\fZ_k}\tilde{\fZ}_k^{\T}$ with $\tilde \fZ_k\in\R^{\stateDim\times\stateDim_k}$. In \cite[Sec.~3]{ManU26}, the stopping criteria for line 5 and 6 of \Cref{alg:statIterGLE}, are defined on the base of a suitably derived error bound for $\|\fX-\tilde{\fX}_k\|_2$. Here, we introduce a different stopping criteria with the aim of computing an approximation $\tilde{\fX}_k$ that, for a given $\tol>0$, satisfies the relative error condition
\begin{equation*}
\frac{\|\fX-\tilde{\fX}_k\|_2}{\|\fX \|_2}\;\le\;\tol.
\end{equation*}
Having a stopping criteria based on relative errors makes the convergence of \Cref{alg:statIterGLE} more robust for floating-point arithmetic since numbers are stored with a fixed number of significant digits.
\subsubsection{Relative error bound for approximate solution of the Lyapunov matrix equation}
Assume that $\fX_{\indKS}\in\Spsd{\stateDim}$ is the approximation of the solution $\fX_k\in\Spsd{\stateDim}$ of the Lyapunov equation~\eqref{alg:statIterGLE:lyap} in the $\indKS$th step (and $k$th iteration of \Cref{alg:statIterGLE}) of an iterative procedure that construct a sequence of approximations converging to the exact solution. We denote by
\begin{equation}
	\label{eqn:lyap:residual}
	\fR_\indKS \vcentcolon= \fA\fX_\indKS + \fX_\indKS\fA^\T + {\fB}_k{\fB}_k^\T
\end{equation}
the corresponding residual and obtain the following residual based relative error relation.

\begin{proposition}[Relative error bound for approximate solution of Lyapunov equations]
	\label{prop1}
	Consider the Lyapunov equation~\eqref{alg:statIterGLE:lyap} with $\fA$ Hurwitz, and assume that $\fX_\indKS\in\Spsd{\stateDim}$ is an approximation of the unique solution $\fX_k\in\Spsd{\stateDim}$ of~\eqref{alg:statIterGLE:lyap}. Let $\fE_\indKS \vcentcolon= \fX_\indKS-\fX_k$ and assume that 
	\begin{equation}\label{eqn:ass:prop}
	\|\fX_{\indKS}\|\;\geq\;\frac{\|\fR_\indKS\|_{\Frob}}{\sigma_{\min}(\lyapOper)}
	\end{equation}
	holds, where $\sigma_{\min}(\lyapOper)$ denotes the smallest singular value of $\lyapOper$. Then
	\begin{equation}
		\label{eqn:lyap:errorResidual}
		\frac{\|\fE_\indKS\|_2}{\|\fX_k\|_2} \;\leq\; \frac{\|\fR_\indKS\|_{\Frob}}{ \sigma_{\min}(\lyapOper)\|\fX_{\indKS}\|_2-\|\fR_{\indKS}\|_{\Frob}}.
	\end{equation}
\end{proposition}

\begin{proof}
	By \cite[Prop.~3.3]{ManU26} we have
	\begin{equation}\label{eqn:prop3.3}
		\|\fE_\indKS\|_2 \;\leq\; \frac{\|\fR_\indKS\|_{\Frob}}{\sigma_{\min}(\lyapOper)};
	\end{equation}
	and by the Weyl theorem for singular values, we get
	\begin{equation}\label{eqn:den}
   \|\fX_k\|_2\;=\;\|\fX_{\indKS}-\fE_{\indKS}\|_2\;\geq\;{\|\fX_{\indKS}\|_{2}-\|\fE_{\indKS}\|_{2}}\;\geq\;\|\fX_{\indKS}\|_{2}-\frac{\|\fR_\indKS\|_{\Frob}}{\sigma_{\min}(\lyapOper)}\;\geq\;0,
	\end{equation}
	where we used \eqref{eqn:prop3.3} and assumption \eqref{eqn:ass:prop}. Combining \eqref{eqn:prop3.3} and \eqref{eqn:den} to upper bound the relative error ${\|\fE_\indKS\|_2}{\|\fX_k\|^{-1}_2}$, we get \eqref{eqn:lyap:errorResidual}.
\end{proof}

The construction of the approximation $\fX_ \indKS$ via Krylov subspace methods is described in detail in \cite[Sec.~3.2.1]{ManU26}. These projection techniques are essential for solving large-scale Lyapunov equations efficiently; specifically, they enable the evaluation of the residual norm in \eqref{eqn:lyap:errorResidual} with a computational complexity of only $\calO(\stateDim)$ floating point operations and do not require storing full matrices of dimension $\stateDim$. Note furthermore that \eqref{eqn:ass:prop} can be viewed as a generic assumption since the residual will eventually converge towards zero as the dimension of the projection subspaces increases; therefore, \eqref{eqn:ass:prop} will be satisfied for some sufficiently large $\indKS^{\star}\in\N$.

From a computational perspective, the key challenge lies in evaluating $\sigma_{\min}(\lyapOper)$, since this typically requires at least $\calO(\stateDim^2)$ floating point operations. In \cite[Sec.~3.3.2]{ManU26}, the authors provide a detailed explanation of the conditions under which the computation of $\sigma_{\min}(\lyapOper)$ can be reduced to only $\calO(\stateDim)$ operations, owing to specific properties of $\fA$.

\subsubsection{Stopping criteria for \GLE based on a relative error bound}\label{subsubsec:exit:tol}
First, let us present the following auxiliary result.
\begin{proposition}
	\label{prop:InexactGramians}
	Suppose \Cref{ass:statIterGLE} is satisfied, and let $\fX$ and $\fX_k$ denote the unique solution of the \GLE~\eqref{eqn:generic:GLE:2} and its approximation at the $k$th iteration computed through \Cref{alg:statIterGLE}, respectively, where the Lyapunov equation~\eqref{alg:statIterGLE:lyap} is solved exactly. Then
	\begin{equation}\label{eq:56}
		\fX\; \succeq \;\fX_k\,, \quad\text{for all } k\in\N.
	\end{equation}
\end{proposition}

\begin{proof}
	 Let $\fE_k=\fX-\fX_k$. Clearly, \eqref{eq:56} holds if and only if $\fE_k\succeq 0$ holds for all $k\in\N$. First, observe that $\fE_0\succeq 0$ since \Cref{alg:statIterGLE} is (implicitly) initialized with $\fX_0=\zeroMat$ and $\fX$ is symmetric positive semidefinite. Next, let us assume $\fE_{k-1}\succeq 0$ for some $k\in\N$. By subtracting from \eqref{eqn:generic:GLE:2} the equation
     \begin{equation}\label{eq30}
			\fA\fX_k + \fX_k\fA^\T + \sum_{j=1}^{M}\left(\fN_j\fX_{k-1}\fN_j^\T\right) + \fB\fB^\T = \zeroMat,
		\end{equation}
 which corresponds to \eqref{alg:statIterGLE:lyap} in \Cref{alg:statIterGLE}, we obtain
\begin{equation*}
	 \fA\fE_k+\fE_k\fA^\T+\sum_{j=1}^{M}\left(\fN_j\fE_{k-1}\fN_j^\T\right)=\zeroMat,
\end{equation*}
and since $\fE_{k-1}\succeq 0$ by assumption, it directly follows that $\fE_k\succeq 0$ by the fact that $\fA$ is asymptotically stable. Then, \eqref{eq:56} follows by the induction principle.
	\end{proof}
With this preparation, we can now derive the relative error bound that we use to formulate the stopping criterion in \Cref{alg:statIterGLE}.

\begin{theorem}[Relative error bound for the approximate solution of \GLE]
	\label{teo1}
	Suppose \Cref{ass:statIterGLE} holds. Let $\fX=\fZ\fZ^\T$ with $\fZ\in\R^{\stateDim\times\stateDim}$ be the unique solution of the \GLE \eqref{eqn:generic:GLE:2} and $\tilde{\fX}_k=\tilde{\fZ}_k\tilde{\fZ}_k^\T$, with $\tilde{\fZ}_k\in\R^{\stateDim\times\stateDim_k}$, be the approximate solution of the \GLE computed at iteration $k$ of \Cref{alg:statIterGLE} applying the subspace projection method to solve the Lyapunov equation in~\eqref{alg:statIterGLE:lyap}. Then
	\begin{align}
		\label{eq:50}
		\frac{\|\fX-\tilde{\fX}_k\|_2}{\|\fX\|_2}\; \le\; \frac{\sigma_{\min}(\lyapOper)\gamma \Delta_{k,k-1}+ (1+\gamma)\|\fR_k\|_{\Frob}+\gamma\|\fR_{k-1}\|_{\Frob}}{\sigma_{\min}(\lyapOper)\|\tilde \fX_{k}\|_{2}-\|\fR_k\|_{\Frob}},
	\end{align}
	where $\fR_k$ and $\fR_{k-1}$ are the residuals corresponding to the approximation of the Lyapunov equation~\eqref{alg:statIterGLE:lyap} at the $k$th and $(k-1)$th iterations of \Cref{alg:statIterGLE} due to the subspace approximation, respectively, $\gamma$ is defined as
    \begin{equation*}
		\gamma \;\vcentcolon=\; \frac{\|\calL^{-1}\gleOper\|}{1-\|\calL^{-1}\gleOper\|},
	\end{equation*}
    and
	\begin{equation}
		\Delta_{k,k-1}\; \vcentcolon=\; \|(\tilde \fZ_k-\tilde \fZ_{k-1})\tilde \fSigma_k^{\frac{1}{2}}\|_\Frob+\|(\tilde\fZ_k-\tilde\fZ_{k-1})\tilde\fSigma_{k-1}^{\frac{1}{2}}\|_\Frob,\quad\fZ_k\;=\;\tilde \fU_k\tilde\fSigma^{\frac{1}{2}}_k;
	\end{equation}
    being $\tilde\fSigma_k\in\R^{\stateDim_k\times\stateDim_k}$ a diagonal matrix for all $k\in\N$.
\end{theorem}

\begin{proof}
	The upper bound for $\|\fX-\tilde\fX_k\|_2$ is given by \cite[Thm.~3.6]{ManU26}. For the lower bound on $\|\fX\|_2$, we first note from \Cref{prop:InexactGramians} that $\|\fX\|_2\geq\|\fX_k\|_2$, and the lower bound for $\|\fX_k\|_2$ comes via \eqref{eqn:den} and \eqref{eqn:ass:prop}. Combining these bounds and rearranging the resulting inequality yields \eqref{eq:50}.
\end{proof}
To use \eqref{eq:50} as the stopping criterion for \Cref{alg:statIterGLE}, we first refer to the discussion in \cite[Sec.~3.2.2]{ManU26}, which allows us to set $\gamma=1$. Given this, thanks to \Cref{teo1}, if we require
\begin{align*}
	\|\fR_k\|_\Frob \le \frac{\zeta_1\sigma_{\min}(\lyapOper)\|\tilde \fX_k\|_2\tol}{3+\zeta_1\tol}, \quad
\Delta_{k,k-1}\le 	\frac{\zeta_2\tol\left(\sigma_{\min}(\lyapOper)\|\tilde\fX_k\|_2-\|\fR_k\|_{\Frob}\right)}{\sigma_{\min}(\lyapOper)},
\end{align*}
with $\zeta_1+\zeta_2=1$ and $\zeta_1$, $\zeta_2>0$, during \Cref{alg:statIterGLE}, we are guaranteed to obtain $\|\fX-\fX_k\|_2\|\fX\|^{-1}_2\le \tol$. 
%-----------------------------------------------------------------------------%

%-----------------------------------------------------------------------------%
\section{Numerical Experiments}
\label{sec:examples}
In this section, we present numerical experiments to assess the effectiveness of our method. We first examine a \SDAE derived from a constrained mass--spring--damper system. Next, to highlight the generality of the proposed framework, we study an artificial \SDAE that alternates between the constrained mass--spring--damper system and the semidiscretized unsteady Stokes equations. Using these benchmark problems, we illustrate that the developed \MOR technique, together with the numerics for solving the \GLE with relative error control, performs effectively.

All calculations were performed with \matlab~2024b on a MacBook Pro with an Apple M2 Pro processor and 16GB of RAM. 

\vspace{0.2cm}
\noindent\fbox{%
	\parbox{0.98\textwidth}{%
		The code and data used to generate the subsequent results are accessible via
		\begin{center}
			\url{https://zenodo.org/records/20056380}
		\end{center}
		under MIT Common License.
	}%
}\\[.2em]
%-----------------------------------------------------------------------------%
\subsection{Constrained mass--spring--damper system} 
\label{subsec:MSD}
We consider the holonomically constrained mass-spring-damper system presented in \cite[Sec.~4]{MehS05} with $g\in\N$ masses. The vibration of this system is described by the descriptor system
\begin{align}\label{eq34}
	\left\{\quad
	\begin{aligned}
		\dot{\fp}(t) &= \fv(t), & 
		\fM\dot{\fv}(t) &=\fK\fp(t)+\fD\fv(t)-\fG^\T\boldsymbol{\lambda}(t)+\fB_2\inp(t),\\
		\zeroVec &= \fG\fp(t),\\
		\out(t) &= \fC_1\fp(t),
	\end{aligned}\right.
\end{align}%
where $\fp(t) \in \R^g$ is the position vector, $\fv(t) \in \R^g$ is the velocity vector, $\flambda(t)\in \R$ is the Lagrange multiplier, $\fM = \text{diag}(m_1,\ldots,m_g)$ is the mass matrix, $\fD$ and $\fK$ are the tridiagonal damping and stiffness matrices, $\fG = [1, 0, \ldots,0, -1]\in \R^g$ is the constraint matrix, $\fB_2 = \fe_1$, and $\fC_1 = [ \fe_1, \fe_2, \fe_{g-1} ]^\T$, where $\fe_j$ denotes the $j$th column of the identity matrix $\fI_{g}$. The descriptor system arising from \eqref{eq34} is of index $\nu=3$, and its associated matrices read:
\begin{align}
	\label{eq:ref:con:mas:spr}
	\fE &\vcentcolon= \begin{bmatrix}
		\fI_g & \zeroMat & \zeroMat \\
		\zeroMat & \fM & \zeroMat \\
		\zeroMat & \zeroMat & \zeroMat
	\end{bmatrix}, &
	\fA &\vcentcolon= \begin{bmatrix}
		\zeroMat & \fI_g & \zeroMat \\
		\fK& \fD & -\fG^\T\\
		\fG & \zeroMat & \zeroMat \\
	\end{bmatrix}, &
	\fB &\vcentcolon= \begin{bmatrix}
		\zeroVec \\
		\fB_2\\
		\zeroVec \\
	\end{bmatrix}, &
	\fC\vcentcolon=  \begin{bmatrix}\fC_1^\T\\
	\zeroMat\\
	\zeroMat\\
	\end{bmatrix}^\T .
\end{align}
The state dimension of the system is $\stateDim=2g+1$, while the input and output dimensions are $\inpDim=1$ and $\outDim=3$. The specific setting of parameters is taken from \cite[Sec.~4]{MehS05}. Note that the matrices in \eqref{eq:ref:con:mas:spr} are sparse, and the kernel of $\fE$ is of dimension one. To obtain a switched system, we sample the matrices 
\begin{align}
	\label{eq:swt:mat}
	\fE_j &\vcentcolon= \fE+\calU_{[0,1]}\fE, &
	\fA_j &\vcentcolon= \fA + \begin{bmatrix}
		\zeroMat & \zeroMat & \zeroMat \\
		\zeroMat &\calU_{[0,0.35]} \fI_g & \zeroMat\\
		0.5{\fe_{j+1}^{\T}} & \zeroMat & \zeroMat \\
	\end{bmatrix}, &
	\fB_j &\vcentcolon= \fB +\begin{bmatrix}
		 \fe_j\\
		 \fe_j\\
		 \zeroMat
	\end{bmatrix}
\end{align}
for $j=2,\ldots,M$, with $j=1$ given by \eqref{eq:ref:con:mas:spr}, and where $\calU_{[a,b]}$ denotes the uniform distribution in the interval $[a,b]$. With these choices, all finite eigenvalues of $(\fE_j,\fA_j)$ have a negative real part; moreover, the impulsive contribution of the output function that originates from the differential part of the state is zero because $\calC_{i,j}$, as defined in \eqref{eqn:reformulation:SDAE}, vanishes for all $i,j = 1,\ldots,M$. 

We consider $M=5$ modes, and after decoupling each \DAE mode and thus reformulating the \SDAE as a switched system with jumps and impulses, we set the exit tolerance~$\tol$ for \Cref{alg:statIterGLE} to $10^{-12}$. We emphasize that the matrices remain sparse even after decoupling, so no implicit definition via matrix-vector products is needed for the efficient approximation of the \GLEs. Observe that, in order to accelerate the computations, we do not explicitly evaluate $\sigma_{\min}(\lyapOper)$, but instead approximate it by twice the smallest singular value of the matrix associated with the operator $\lyapOper$. This approximation is mathematically rigorous only when the underlying matrix is symmetric. Nevertheless, for the purpose of demonstrating the effectiveness of the proposed reduction procedure on a relatively large switched system, we adopt this approximation as a pragmatic choice. 

Once the approximate solutions of the \GLEs have been computed, we perturb them as in \cite[Sec.~4.4]{ManU26}. This procedure tightens the error bound \eqref{eq:err:est:3} by reducing (or even neglecting) the contribution arising from the inexact \LMIs. We next assess the accuracy of the resulting \ROM for $\stateDim=20$ by comparing it against the \FOM for $\stateDim=10001$ for the input functions
\begin{equation}\label{eqn:inp:CMSS}
	\inp_1(t) = 1
	\quad\text{and}\quad 	
	\inp_2(t) = \sin\left(2\pi\ee^\frac{t}{8}\right).
\end{equation}
In both cases, the switching signal is selected randomly. Denoting by $y_{\switch_k,i}$, for $i=1,\ldots,\outDim\nu_{\switch_{k}}$, the $i$-th component of the output vector $\out_{\switch_k}$, \Cref{fig1a} displays the results for the first input function and the second output component, while \Cref{fig1b} shows the values corresponding to the first output component and the second input function.
\begin{figure}[t]
	\centering
	\begin{subfigure}[t]{.45\linewidth}
		\input{img/CMSS_Plot_1.tex}
		\caption{$\inp_1(t)= 1$, $y_{\switch(t),2}(t)$ (\FOM), and $\bar{y}_{\switch(t),2}(t)$ (\ROM).}
		\label{fig1a}
	\end{subfigure}
	\hfil
	\begin{subfigure}[t]{.45\linewidth}
		\input{img/CMSS_Plot_2.tex}
		\caption{$\inp_2(t) =\sin(2\pi(\ee^{\frac{t}{8}}))$, $y_{\switch(t),1}(t)$ (\FOM), and $\bar{y}_{\switch(t),1}(t)$ (\ROM). Zoom on the top right corner to highlight the jumps.}
		\label{fig1b}
	\end{subfigure}
	\caption{Constrained mass-spring-damper system with $M=5$ modes. Qualitative comparison of the second (left) and first (right) output entry of the \FOM, with $\stateDim=10001$, and the \ROM, with $\stateDimRed=20$, for two different switching signals and input functions.}
	\label{fig1}%
\end{figure}
The plots illustrate that the \ROM obtained through the \PBR for \SDAE closely matches the \FOM in both cases. Furthermore, computing the outputs with the \ROM required only $0.17$ and $0.76$ seconds for $\inp_1$ and $\inp_2$, respectively, whereas the \FOM took $45$ and $741$ seconds. This corresponds to a speed-up factor between $280$ and $945$.

We next consider the accuracy in terms of the \ROM dimension. Recall that for a \ROM of dimension $\stateDimRed$, the associated error bound $\tau(\stateDimRed,\inp)$ is given in \eqref{eq:err:est:3}; in addition, let us define
\begin{align}
\eta(\inp)\;\vcentcolon=&\;\left(\int_{0}^{t}\|\inp(s)\|_2^{2}\,\ds\;+\;\sum_{k=1}^{K}\|\fU_{\switch_k}(t_k^{-})\|_2^2\right)^{\frac{1}{2}},\label{eqn:eta}\\
	\varepsilon(\stateDimRed,\inp)\;\vcentcolon=&\;{\left(\int_{0}^{t}\|\out_{\switch(s)}(s)-\bar \out_{\switch(s)}(s)\|^2_2\,\ds\right)^{\frac{1}{2}}}\eta(\inp)^{-1}\label{eqn:sys:err:b},\\
    \tau_{rel}(\stateDimRed,\inp)\;\vcentcolon=&\;\tau(\stateDimRed,\inp)\eta(\inp)^{-1}\label{eqn:sys:err:c}.
\end{align}

\begin{figure}[t]	
		\centering
\begin{tikzpicture}

\begin{axis}[%
width=\imageWidth,
height=\imageHeight,
at={(1.011in,0.642in)},
scale only axis,
grid=both,
grid style={line width=.1pt, draw=gray!10},
major grid style={line width=.2pt,draw=gray!50},
axis lines*=left,
axis line style={line width=\lineWidth},
mark size=2.2pt,
xmin=0,
xmax=55,
xlabel style={font=\color{white!15!black}},
xlabel={$\stateDimRed$},
ymode=log,
ymin=3e-15,
ymax=1e6,
yminorticks=true,
axis background/.style={fill=white},
legend style={%
	legend cell align=left, 
	align=left, 
	font=\tiny,
	draw=white!15!black,
	at={(0.98,0.98)},
	anchor=north east,},
]
\addplot [color=mycolor1, line width=\lineWidth, mark=o, mark options={solid, mycolor1}]
table [x index=0, y index=1, col sep=comma]
{img/DataCSV/CMSS_2a.csv};
\addlegendentry{$\varepsilon(\stateDimRed,\inp_1)$}

\addplot [color=mycolor2, dashed, line width=\lineWidth, mark=+, mark options={solid, mycolor2}]
table [x index=0, y index=2, col sep=comma]
{img/DataCSV/CMSS_2a.csv};
\addlegendentry{$\varepsilon(\stateDimRed,\inp_2)$}

\addplot [color=mycolor3, line width=\lineWidth, mark=o, mark options={solid, mycolor3}]
 table [x index=0, y index=3, col sep=comma]
{img/DataCSV/CMSS_2a.csv};
\addlegendentry{$\tau_{rel}(\stateDimRed,\inp_1)$}

\addplot [color=mycolor4, dashed, line width=\lineWidth, mark=+, mark options={solid, mycolor4}]
 table [x index=0, y index=4, col sep=comma]
{img/DataCSV/CMSS_2a.csv};
\addlegendentry{$\tau_{rel}(\stateDimRed,\inp_2)$}

\end{axis}
\end{tikzpicture}%
	\caption{Switched constrained mass-spring-damper system with $M=5$ modes. Decay of the reduction error \eqref{eqn:sys:err:b} and of the error estimate \eqref{eqn:sys:err:c} with respect to the \ROM dimension $\stateDimRed$ for the input functions in \eqref{eqn:inp:CMSS}.} 
    \label{fig2a}
\end{figure}\noindent
In \Cref{fig2a}, we show the decay of the reduction error \eqref{eqn:sys:err:b} along with that of the rescaled error estimate for both input functions given in \eqref{eqn:inp:CMSS}. We see that the error bound indeed lies above the actual computed error; nevertheless, for $\inp_2$, there is a discrepancy of several orders of magnitude between them. This gap is caused by the large magnitude of the derivatives of $\inp_2$, which enter the definition of $\eta(\inp)$ (see \eqref{eqn:eta}) through \eqref{eqn:def:jum:imp:c}. 

\subsection{An artificial \SDAE} 
\label{subsec:Stokes}
As a final numerical example, to illustrate the capability of our \MOR to handle the general class of \SDAE, we examine a \SDAE featuring two entirely unrelated models. The first one corresponds to the constrained mass-spring-damper system discussed in the previous section, while the second is obtained from a semidiscretization of the Stokes equations. The Stokes equations describe the flow of fluids at very low velocities without convection and coincide with the linearization of the Navier-Stokes equations around the zero-state. After a semi-discretization in space (see \cite{morwiki_stokes} which is based on \cite{Sch07}), we obtain the descriptor system
\begin{align}\label{eq42}
	\left\{\quad
	\begin{aligned}
		\dot{\fv}(t) &= \fA_{11}\fv(t) + \fA_{12}\frho(t) + \fB_1\inp(t),\\
		\zeroVec &= \fA^\T_{12}\fv(t) + \fB_2\inp(t),\\
		\out(t) &= \fC_1\fv(t) + \fC_2\frho(t), 
	\end{aligned}\right.
\end{align}
where $\fv(t)\in\R^{\stateDim_{\fv}}$ and $\frho(t)\in\R^{\stateDim_{\frho}}$ are the semidiscretized vectors of velocities and pressures, respectively. The \DAE~\eqref{eq42} has index two, and the order $\stateDim = \stateDim_{\fv} + \stateDim_{\frho}$ of system \eqref{eq42} depends on the fineness of the discretization and is usually very large, whereas the number $\inpDim$ of inputs and the number $\outDim$ of outputs are typically small. The system \eqref{eq42} can be expressed in the standard \DAE~form as
\begin{align}
	\label{eqn:ref:mat:ST}
	\fE &= \begin{bmatrix}
		\fI & \zeroMat \\
		\zeroMat & \zeroMat
	\end{bmatrix}, &
	\fA &= \begin{bmatrix}
		\fA_{11} & \fA_{12} \\
		\fA_{12}^\T & \zeroMat
	\end{bmatrix}, &
	\fB &= \begin{bmatrix}
		\fB_{1} \\
		\fB_{2}
	\end{bmatrix}, &
	\fC &= \begin{bmatrix}
		\fC_{1} & \fC_{2}
	\end{bmatrix}.
\end{align}
Thus, our \SDAE is obtained by alternating between \eqref{eq:ref:con:mas:spr} and \eqref{eqn:ref:mat:ST} with $\inpDim=1$ and $\outDim=3$; that is, between two \DAE of different indices, each describing a distinct physical phenomenon. We consider $\stateDim=1159$ for both systems, although modes of different sizes could be considered, provided that suitable transition matrices are available. This, after decoupling via \QWF, leads to $\stateDim_{\fJ_1}=1156$ (the constrained mass-spring-damper system) and $\stateDim_{\fJ_2}=361$ (for the semidiscretized Stokes). After approximating the solution of \eqref{eqn:GLE:c}, we compare the \FOM with the \ROM resulting for $\stateDimRed=20$ by considering the two input functions defined in \eqref{eqn:inp:CMSS} multiplied by the norm of $\fB$ in \eqref{eqn:ref:mat:ST}. We perform this rescaling of the input function since we noticed that normalizing the input matrix when it has a large norm has a positive effect on tightening the \LMIs. If the input matrix is normalized, to avoid altering the input-output map, one must multiply the input function by the original norm of the input matrix. 
\begin{figure}[t]
	\centering
	\begin{subfigure}[t]{.47\linewidth}
		\input{img/CMSS-ST_Plot_1a.tex}
		\caption{$\inp_1(t)= 1$, ${y}_{\switch(t),3}(t)$ (\FOM), and $\bar{y}_{\switch(t),3}(t)$ (\ROM).}
		\label{fig3a}
	\end{subfigure}
    \hfil
	\begin{subfigure}[t]{.47\linewidth}
		\input{img/CMSS-ST_Plot_1b}
		\caption{$\inp_2(t) =\sin(2\pi(\ee^{\frac{t}{8}}))$, $y_{\switch(t),1}(t)$ (\FOM), and $\bar{y}_{\switch(t),1}(t)$ (\ROM).}
		\label{fig3b}
	\end{subfigure}
	\caption{Artificial \SDAE example. Qualitative comparison of the third (left) and first (right) output entry of the \FOM, with $\stateDim=1159$, and the \ROM, with $\stateDimRed=20$, for two different switching paths and input signals. A magnified view is included in both figures to emphasize the accuracy of the approximation.}
	\label{fig3}%
\end{figure}

The switching sequence is again chosen with random switching times. The results are shown in \Cref{fig3a} and \ref{fig3b}, where the \ROM with $\stateDim=20$ closely matches the \FOM. We also note that computing the \ROM for both input functions was approximately $100$ times faster than the corresponding \FOM. 

\Cref{fig4a} shows how the error \eqref{eqn:sys:err:b} and the rescaled error bound \eqref{eqn:sys:err:c} decrease as functions of $\stateDimRed$ for both input functions. About \Cref{rmk-l2-linf}, we note that although the proposed methodology is specifically designed to bound the $L_2$-norm of the augmented output function, in our experiments, we did not observe larger errors at the switching times in practice.
\begin{figure}[t]
	\centering
    \begin{subfigure}[t]{.47\linewidth}
\begin{tikzpicture}

\begin{axis}[%
width=\imageWidth,
height=\imageHeight,
at={(1.011in,0.642in)},
scale only axis,
grid=both,
grid style={line width=.1pt, draw=gray!10},
major grid style={line width=.2pt,draw=gray!50},
axis lines*=left,
axis line style={line width=\lineWidth},
mark size=2.2pt,
xmin=0,
xmax=30,
xlabel style={font=\color{white!15!black}},
xlabel={$\stateDimRed$},
ymode=log,
ymin=5e-7,
ymax=1e2,
yminorticks=true,
axis background/.style={fill=white},
legend style={%
	legend cell align=left, 
	align=left, 
	font=\tiny,
	draw=white!15!black,
	at={(0.98,0.98)},
	anchor=north east,},
]
\addplot [color=mycolor1, line width=\lineWidth, mark=o, mark options={solid, mycolor1}]
table [x index=0, y index=1, col sep=comma]
{img/DataCSV/CMSS_3c.csv};
\addlegendentry{$\varepsilon(\stateDimRed,\inp_1 )$}

\addplot [color=mycolor2, dashed, line width=\lineWidth, mark=+, mark options={solid, mycolor2}]
table [x index=0, y index=3, col sep=comma]
{img/DataCSV/CMSS_3c.csv};
\addlegendentry{$\varepsilon(\stateDimRed, \inp_2)$}

\addplot [color=mycolor3, line width=\lineWidth, mark=o, mark options={solid, mycolor3}]
 table [x index=0, y index=5, col sep=comma]
{img/DataCSV/CMSS_3c.csv};
\addlegendentry{$\tau_{rel}(\stateDimRed,\inp_1)$}

\addplot [color=mycolor4, dashed, line width=\lineWidth, mark=+, mark options={solid, mycolor4}]
 table [x index=0, y index=7, col sep=comma]
{img/DataCSV/CMSS_3c.csv};
\addlegendentry{$\tau_{rel}(\stateDimRed,\inp_2)$}

\end{axis}
\end{tikzpicture}%
		\caption{Decay of the reduction error \eqref{eqn:sys:err:b} and of the error estimate \eqref{eqn:sys:err:c} with respect to the \ROM dimension $\stateDimRed$ for the input functions in \eqref{eqn:inp:CMSS}.}
		\label{fig4a}
	\end{subfigure}
    \hfil
	\begin{subfigure}[t]{.47\linewidth}
\begin{tikzpicture}

\begin{axis}[%
width=\imageWidth,
height=\imageHeight,
at={(1.011in,0.642in)},
scale only axis,
grid=both,
grid style={line width=.1pt, draw=gray!10},
major grid style={line width=.2pt,draw=gray!50},
axis lines*=left,
axis line style={line width=\lineWidth},
mark size=2.2pt,
xmin=0,
xmax=30,
xlabel style={font=\color{white!15!black}},
xlabel={$\stateDimRed$},
ymode=log,
ymin=8e-7,
ymax=3e-2,
yminorticks=true,
axis background/.style={fill=white},
legend style={%
	legend cell align=left, 
	align=left, 
	font=\tiny,
	draw=white!15!black,
	at={(0.98,0.98)},
	anchor=north east,},
]
\addplot [color=mycolor1, line width=\lineWidth, mark=o, mark options={solid, mycolor1}]
table [x index=0, y index=1, col sep=comma]
{img/DataCSV/CMSS_3c.csv};
\addlegendentry{$\varepsilon(\stateDimRed,\inp_1 )$}

\addplot [color=mycolor2, dashed, line width=\lineWidth, mark=+, mark options={solid, mycolor2}]
table [x index=0, y index=2, col sep=comma]
{img/DataCSV/CMSS_3c.csv};
\addlegendentry{$\varepsilon(\stateDimRed, \inp_1)$ (no imp. )}

\addplot [color=mycolor3, line width=\lineWidth, mark=o, mark options={solid, mycolor3}]
 table [x index=0, y index=3, col sep=comma]
{img/DataCSV/CMSS_3c.csv};
\addlegendentry{$\varepsilon(\stateDimRed,\inp_2)$}

\addplot [color=mycolor4, dashed, line width=\lineWidth, mark=+, mark options={solid, mycolor4}]
 table [x index=0, y index=4, col sep=comma]
{img/DataCSV/CMSS_3c.csv};
\addlegendentry{$\varepsilon(\stateDimRed,\inp_2)$ (no imp. )}

\end{axis}
\end{tikzpicture}%
		\caption{Comparison between the proposed \ROM and the one obtained without including the impulsive matrices in the \GLEs}
		\label{fig4b}
	\end{subfigure}
	\caption{Artificial \SDAE example. Decay of the associated error and error bound (left) and error comparison with \ROM obtained without including the impulsive matrices in the \GLEs.}
	\label{fig4}%
\end{figure}
Finally, in order to showcase the importance of including the impulsive components of the switched system via the augmented input and output matrices, we display in \Cref{fig4} the comparison between the errors obtained for the two switching signals and input functions using the \ROM obtained from approximating the solutions of \eqref{eqn:GLE:c} and a \ROM obtained via \GLEs that only accounts for $\tilde \fB_j $ and $\tilde \fC_j$, for $j=1,\ldots,M$. The plot clearly demonstrates that incorporating the impulsive components often results in reduced error and does not lead to any deterioration in accuracy.

%-----------------------------------------------------------------------------%
\section{Conclusions and perspectives}

In this work, we introduce a novel projection-based \MOR method for \SDAE equipped with a posteriori error bounds to assess the quality of the approximation. The proposed \MOR method is based on the \PBR for standard switched systems introduced in \cite[Sec.~4]{ManU26}. To this end, and after providing the necessary theoretical justification, we reformulate the \SDAE as an \SLS with jumps at the switching times and with augmented input and output matrices; see \Cref{sec:MORforDAE}. 

Since the \PBR for \SDAE relies on the solution of several pairs of \GLEs, we employ the solver described in \cite[Sec.~3]{ManU26} and additionally derive new stopping criteria that allow us to rigorously bound the relative error in the approximation of the \GLE solution; see \Cref{subsubsec:exit:tol}. 

The effectiveness and accuracy of the proposed method are demonstrated through numerical experiments. In particular, we show that the proposed methodology can handle systems of \SDAE in which each mode corresponds to physical models of completely different natures; see \Cref{subsec:Stokes}.

A bottleneck of the proposed method is its dependence on decomposing each \DAE mode into separate differential and impulsive subsystems. Efficiently obtaining such a decomposition for broad classes of \DAEs, while preserving sparsity properties when they are present, was not addressed in this work and thus remains a clearly relevant and interesting open research problem. Furthermore, the effective use of stopping criteria based on the relative-error upper bound is hindered by the need to compute the smallest singular value of the Lyapunov operator. A potential remedy that could be investigated consists of exploiting efficient $\varepsilon$-pseudospectrum evaluation techniques to derive error bounds that circumvent the computation of the smallest singular value of $\stateDim^2 \times \stateDim^2$ matrices; see \cite{ManMG24} and the discussion in \cite[App.~A]{ManU26}.
	
%-----------------------------------------------------------------------------%
\section*{Acknowledgment}
MM acknowledges funding by the BMBF (grant no.~05M22VSA) and support by the Stuttgart Center for Simulation Science (SimTech).
BU is partly funded by the Deutsche Forschungsgemeinschaft (DFG, German Research Foundation) -- Project-ID 258734477 -- SFB 1173.

%-----------------------------------------------------------------------------%
\bibliographystyle{plain-doi}
\bibliography{journalAbbr,literature}

\end{document}